\documentclass{amsart}

\usepackage[left=25mm,right=25mm,top=25mm,bottom=25mm]{geometry}
\usepackage{graphicx}
\usepackage{bm}
\usepackage[
	pdftitle={Pruned Hurwitz numbers},
	pdfauthor={Norman Do and Paul Norbury},
	ocgcolorlinks,
	linkcolor=linkblue,
	citecolor=linkred,
	urlcolor=linkred]
{hyperref}
\usepackage{color}
\definecolor{linkred}{rgb}{0.75,0,0}
\definecolor{linkblue}{rgb}{0,0,0.75}
\usepackage{mathpazo}
\usepackage{microtype}
\usepackage{multicol}
\usepackage{booktabs}
\usepackage{latexsym}
\usepackage{multirow}
\usepackage{setspace}

\theoremstyle{plain}
\newtheorem{theorem}{Theorem}
\newtheorem{proposition}{Proposition}[section]
\newtheorem{lemma}[proposition]{Lemma}

\newtheorem{corollary}[proposition]{Corollary}

\theoremstyle{definition}
\newtheorem{example}[proposition]{Example}
\newtheorem{definition}[proposition]{Definition}
\newtheorem{remark}[proposition]{Remark}

\newcommand {\dd}{\mathrm{d}}

\newcommand{\cc}{\mathcal{C}}
\newcommand{\ch}{\mathcal{H}}
\newcommand{\ck}{\mathcal{K}}
\newcommand{\bc}{\mathbb{C}}
\newcommand{\bp}{\mathbb{P}}

\newcommand{\br}{\mathbb{R}}
\newcommand{\bz}{\mathbb{Z}}
\newcommand{\modm}{\mathcal{M}}
\newcommand{\fat}{\mathcal{F}\hspace{-.3mm}{\rm at}_{g,n}}

\begin{document}

\title{Pruned Hurwitz numbers}
\author{Norman Do \and Paul Norbury}
\address{School of Mathematical Sciences, Monash University, Australia 3800}
\address{Department of Mathematics and Statistics, The University of Melbourne, Victoria 3010, Australia}
\email{\href{mailto:norm.do@monash.edu}{norm.do@monash.edu}, \href{mailto:pnorbury@ms.unimelb.edu.au}{pnorbury@ms.unimelb.edu.au}}

\thanks{The authors were partially supported by the Australian Research Council grants DP1094328 (PN) and DE130100650 (ND)}
\date{\today}
\begin{abstract}
We define a new Hurwitz problem which is essentially a small core of the simple Hurwitz problem.  The corresponding Hurwitz numbers have simpler formulae, satisfy effective recursion relations and determine the simple Hurwitz numbers.  We also apply this idea of finding a smaller simpler enumerative problem to orbifold Hurwitz numbers and Belyi Hurwitz numbers.
\end{abstract}

\maketitle

\setcounter{tocdepth}{1}
\tableofcontents

\section{Introduction}

In 1891 Hurwitz \cite{HurRie} introduced the problem of enumerating connected branched covers of $\mathbb{CP}^1$ up to isomorphism  with simple ramification over $m$ fixed points in $\mathbb{CP}^1$ and ramification given by a partition $\mu=(\mu_1, \mu_2, \ldots, \mu_n)$ over $\infty\in\mathbb{CP}^1$.  The Riemann-Hurwitz formula shows that the genus of the cover satisfies $m=m(g,\mu)=2g-2+n+|\mu|$ where $|\mu|=\mu_1+\mu_2+\ldots+\mu_n$. Hurwitz described the following equivalent factorisation problem in the symmetric group $S_{|\mu|}$.  We say that a product $\sigma_1\cdot \sigma_2\cdots\sigma_m$ in $S_{|\mu|}$ is {\em transitive} if the collection $\{\sigma_1, \sigma_2, \ldots, \sigma_m\}$ acts transitively on the set  $\{1,2,...,|\mu|\}$.   Given $\mu$, choose $T\in S_{|\mu|}$ of cycle type $\mu=(\mu_1, \mu_2, \ldots, \mu_n)$ and define $H_{g,n}(\mu)$ to be the number of transitive factorisations of $T$ by $m$ transpositions:
\begin{equation}  \label{factor}
\sigma_1\cdot \sigma_2\cdots\sigma_m=T.
\end{equation}
The number $H_{g,n}(\mu)$ is independent of the choice of $T$.  

This problem was studied further by Hurwitz \cite{HurAnz}, Goulden-Jackson \cite{GJaTra} and Ekedahl-Lando-Shapiro-Vainshtein \cite{ELSVHur} where it was shown that:
\begin{equation}  \label{hur}
 \frac{H_{g,n}(\mu)}{m(g,\mu)!}=\prod_{i=1}^n\frac{\mu_i^{\mu_i}}{\mu_i!}P_{g,n}(\mu_1,...,\mu_n)
 \end{equation}
for some polynomial $P_{g,n}(\mu_1,...,\mu_n)$.   The celebrated ELSV formula determines the polynomials $P_{g,n}$ using Hodge integrals over the moduli space of stable curves $\overline{\modm}_{g,n}$.  See Section~\ref{sec:bg}.  

Each number in the set $\{1,2,...,|\mu|\}$ appears in at least one of the factors $\sigma_i$ of  \eqref{factor} since the collection $(\sigma_1, \sigma_2, \cdots ,\sigma_m)$ acts transitively on the set.  In this paper we introduce a new Hurwitz problem with one further condition.  
\begin{definition}
Define the {\em pruned} simple Hurwitz number $K_{g,n}(\mu)$ to be the number of transitive factorisations \eqref{factor} of any $T$ of shape $\mu$ into $m=2g-2+n+|\mu|$ transpositions so that each number $\{1,2,...,|\mu|\}$ appears in at least two of the factors $\sigma_i$.
\end{definition} 
The seemingly innocuous extra condition of each number appearing in at least two factors brings further deep structure to the problem.   The pruned simple Hurwitz number count is a subset of the simple Hurwitz number count and the biggest surprise is that it is extremely well-behaved, and in many ways better behaved than simple Hurwitz numbers.  The pruning condition can also be understood in terms of branched coverings and the word {\em pruned} refers to a graphical description of simple Hurwitz numbers  described in Section~\ref{sec:bg}.   Pruned simple Hurwitz numbers essentially define the core of simple Hurwitz numbers having a formula which is a vast simplification of the formula \eqref{hur} for simple Hurwitz numbers. 
\begin{theorem}  \label{th:main}
The pruned simple Hurwitz numbers satisfy: 
\begin{itemize}
\item[(i)] $K_{g,n}(\mu_1, \mu_2, \ldots, \mu_n)/(2g-2+n+|\mu|)!$ is a polynomial function of the $\mu_i$.
\item[(ii)] $K_{g,n}(\mu_1, \mu_2, \ldots, \mu_n)$ satisfies an effective recursion.
\item[(iii)] $K_{g,n}(\mu_1, \mu_2, \ldots, \mu_n)$ determines and is determined by $H_{g,n}(\mu_1, \mu_2, \ldots, \mu_n)$.
\end{itemize}
\end{theorem}
The recursion in (ii) of Theorem~\ref{th:main} is given explicitly by Proposition~\ref{th:cutandjoin} in Section~\ref{sec:prusim}.  It is not simply the restriction of the cut-and-join recursion for simple Hurwitz numbers because the pruned Hurwitz condition---that each number $\{1,2,...,|\mu|\}$ appears in at least two of the factors $\sigma_i$ of $\mu$---is not preserved under the cut-and-join operations.   The coefficients of top degree terms of $K_{g,n}(\mu_1, \mu_2, \ldots, \mu_n)/m!$ are intersection numbers over $\overline{\modm}_{g,n}$ and this together with the recursion gives a new proof of the Witten--Kontsevich theorem.  The relation (iii) between $K_{g,n}(\mu_1, \mu_2, \ldots, \mu_n)$ and $H_{g,n}(\mu_1, \mu_2, \ldots, \mu_n)$ is given explicitly by \eqref{th:pruning} in Section~\ref{sec:prusim}.    The following example demonstrates the simplification of the formulae for $K_{g,n}(\mu_1, \mu_2, \ldots, \mu_n)$ over $H_{g,n}(\mu_1, \mu_2, \ldots, \mu_n)$.
\begin{example}
For $T\in S_{|\mu|}$ of shape $\mu=(\mu_1,\mu_2,\mu_3)$ consider transitive factorisations by transpositions $\sigma_1\cdot\sigma_2\cdots\sigma_{|\mu|+1}=T$  corresponding to genus 0 branched covers.   The two formulae are
$$H_{0,3}(\mu_1,\mu_2,\mu_3)=(|\mu|+1)!\prod_{i=1}^3\frac{\mu_i^{\mu_i+1}}{\mu_i!},\quad K_{0,3}(\mu_1,\mu_2,\mu_3)=(|\mu|+1)!\mu_1\mu_2\mu_3$$
and the latter is simpler.
\end{example}

For any $a\in\bz^+$ {\em orbifold} Hurwitz numbers $H^{[a]}_{g,n}(\mu)$ are defined as follows.  Given $\mu$, choose $T\in S_{|\mu|}$ of cycle type $\mu=(\mu_1, \mu_2, \ldots, \mu_n)$ and define $H^{[a]}_{g,n}(\mu)$ to be the number of transitive factorisations of $T$ by $m=2g-2+n+\frac{|\mu|}{a}$ transpositions and an element $\sigma_0$ of shape $(a,a,...a)$:
\begin{equation}  \label{factora}
\sigma_0\cdot\sigma_1\cdot \sigma_2\cdots\sigma_m=T.
\end{equation}
In particular $|\mu|$ must be divisible by $a$ to get a non-zero count.  Such factorisations correspond to branched covers of $\bp^1$ and the Riemann-Hurwitz formula in this case shows that the genus of the cover satisfies $m=2g-2+n+\frac{|\mu|}{a}$.  The factor $\sigma_0$ in \eqref{factora} defines a colouring of $\{1,2,...,|\mu|\}$ where each of the $\frac{|\mu|}{a}$ cycles of $\sigma_0$ is given a distinct colour.  Each colour appears in at least one of the transposition factors $\sigma_i$, $i>0$ of  \eqref{factora} since the collection $(\sigma_0,\sigma_1, \sigma_2, \cdots ,\sigma_m)$ acts transitively on $\{1,2,...,|\mu|\}$.  We can generalise the pruning condition as follows.  
\begin{definition}
Define the {\em pruned} orbifold Hurwitz number $K^{[a]}_{g,n}(\mu)$ to be the number of transitive factorisations \eqref{factora} of any $T$ of shape $\mu$ into $m=2g-2+n+|\mu|$ transpositions and an element $\sigma_0$ of shape $(a,a,...a)$ so that each colour---determined by $\sigma_0$---appears in at least two of the transposition factors $\sigma_i$, $i>0$.
\end{definition} 
When $a=1$, this reduces to the pruned simple Hurwitz numbers.  Theorem~\ref{th:main} is a special case of the following theorem.
\begin{theorem}  \label{th:maina}
The pruned orbifold Hurwitz numbers satisfy: 
\begin{itemize}
\item[(i)] $K^{[a]}_{g,n}(\mu_1, \mu_2, \ldots, \mu_n)/(2g-2+n+\tfrac{|\mu|}{a})!$ is a polynomial function of the $\mu_i$.
\item[(ii)] $K^{[a]}_{g,n}(\mu_1, \mu_2, \ldots, \mu_n)$ satisfies an effective recursion.
\item[(iii)] $K^{[a]}_{g,n}(\mu_1, \mu_2, \ldots, \mu_n)$ determines and is determined by $H^{[a]}_{g,n}(\mu_1, \mu_2, \ldots, \mu_n)$.
\end{itemize}
\end{theorem}

The operation of pruning applies to a broader set of combinatorial problems arising from geometry.  It is
related to rational behaviour of a generating function.  Assemble the orbifold Hurwitz numbers into the following generating function.
\begin{equation} \label{orb_gen_func}
\ch^{[a]}_{g,n}(x_1, \ldots, x_n) = \sum_{\mu_1, \ldots, \mu_n = 1}^\infty \frac{H^{[a]}_{g,n}(\mu)}{m!}\frac{x_1^{\mu_1} \cdots x_n^{\mu_n}}{\mu_1\cdots\mu_n}
\end{equation}
Then $\ch^{[a]}_{g,n}(x_1, \ldots, x_n)$ is a convergent power series that extends to a rational function of $z_i$ for $x_i=z_ie^{-z_i^a}$.  A local expansion of $\ch^{[a]}_{g,n}(x_1, \ldots, x_n)$ in $z_i$ yields a generating function for $K^{[a]}_{g,n}(\mu_1, \mu_2, \ldots, \mu_n)$.  The main observation in this paper is that $K^{[a]}_{g,n}(\mu_1, \mu_2, \ldots, \mu_n)$ can be realised as the weighted count of an interesting combinatorial and geometric problem.

Another application of pruning---where expansion in a rational parameter gives rise to an interesting combinatorial or geometric problem---occurs for a different Hurwitz problem known as a Belyi Hurwitz numbers.  Consider connected genus $g$ branched covers $\pi:\Sigma\to S^2$ unramified over $S^2-\{0,1,\infty\}$ with points in the fibre over $\infty$ labeled $(p_1,...,p_n)$ and with ramification $(\mu_1,...,\mu_n)$, ramification $(2,2,...,2)$ over $1$ and arbitrary ramification over $0$.  We call the weighted count of non-isomorphic such branched covers a Belyi Hurwitz number because the covers are known as Belyi maps.  Pruned Belyi Hurwitz covers are Belyi Hurwitz covers with the further restriction that all points above $0$ have non-trivial ramification.  Theorem~\ref{th:main} generalises to this case.

By pruning an enumerative problem, one aim is to produce a simpler core problem to help to understand the original enumerative problem.  Belyi Hurwitz numbers give a good example of a case where the pruned and unpruned versions have independent interest.  The unpruned Belyi Hurwitz numbers arise from discrete surfaces and matrix integral calculations.  Whereas the pruned Belyi Hurwitz covers can be understood as lattice points in the moduli space of curves $\modm_{g,n}$ and give rise to deep information about $\modm_{g,n}$ such as intersection numbers over $\overline{\modm}_{g,n}$ and the orbifold Euler characteristic of $\modm_{g,n}$.

The examples of pruning given here have a further feature in common.  They each satisfy the topological recursion of Eynard and Orantin.  Given a rational curve $C$, for every $(g,n)\in\bz^2$ with $g\geq 0$ and $n>0$ Eynard and Orantin \cite{EOrInv,EOrTop} define a multidifferential, i.e. a tensor product of meromorphic 1-forms on the product $C^n$, denoted by $\omega^g_n(p_1,...,p_n)$ for $p_i\in C$.  When $2g-2+n>0$, $\omega^g_n(p_1,...,p_n)$ is defined recursively in terms of local  information around the poles of $\omega^{g'}_{n'}(p_1,...,p_n)$ for $2g'+2-n'<2g-2+n$.  The generating functions for each of the examples above have been shown to arise as expansions of Eynard-Orantin invariants of particular rational curves.  For simple Hurwitz numbers this was known as the Bouchard-Mari\~no conjecture \cite{BMaHur} proven in \cite{BEMSMat,EMSLap}, a generalisation of this result to orbifold Hurwitz numbers was proven in \cite{DLNOrb, BSLMMir} and Belyi Hurwitz numbers \cite{EOrTop,NorStr}.  A natural question that arises is whether one can apply the idea of pruning to other enumerative problems that satisfy the topological recursion of Eynard and Orantin.  A good candidate is the stationary Gromov-Witten invariants of $\bp^1$ that were proven in \cite{DOSSIde,NScGro}.  We discuss this in Section~\ref{sec:gw}.

We treat the case of pruned simple Hurwitz numbers, in Section~\ref{sec:prusim}, separately from pruned orbifold Hurwitz numbers.  This is due to the independent interest of simple Hurwitz numbers and also because this easier case should help the reader understand the general case of orbifold Hurwitz numbers treated in Section~\ref{sec:pruorb}.

\section{Simple Hurwitz numbers}  \label{sec:bg}

We begin by formally defining simple Hurwitz numbers via simple branched covers.   For $n>0$ and $g\geq 0$ define the set of simple Hurwitz covers:
\begin{align*}
\ch_{g,n}(\mu)=\Big\{f:\Sigma\to S^2\mid& \Sigma \text{ connected genus } g; \text{ with simple ramification over }\{z:z^m=1\};\\
&f^{-1}(\infty)=(p_1,...,p_n) \text{ with respective ramification }\mu=(\mu_1,...,\mu_n);\\
&\hspace{3cm}f\text{ is unramified over } S^2-\{z:z^m=1\}\cup\{\infty\}\Big\} /\sim
\end{align*}
where $m=2g-2+n+|\mu|$ and $\{f_1: \Sigma_1 \to \mathbb{CP}^1\}\sim\{f_2: \Sigma_2 \to \mathbb{CP}^1\}$ if there exists $h: \Sigma_1 \to \Sigma_2$ that satisfies $f_1 = f_2 \circ h$ and preserves the labels over $\infty$.

Define the simple Hurwitz numbers:
\begin{equation}  \label{simhur}
H_{g,n}(\mu_1,...,\mu_n)=\sum_{f\in\ch_{g,n}(\mu)}\frac{\mu_1\cdot...\cdot\mu_n}{|\text{Aut\ }f|}.
\end{equation}
The summands in \eqref{simhur} are integral essentially because the automorphism group is small.  An automorphism of the branched cover $f: (\Sigma; p_1, p_2, \ldots, p_n) \to (\mathbb{CP}^1; \infty)$ is an automorphism $\phi$ of the marked Riemann surface surface $(\Sigma; p_1, p_2, \ldots, p_n)$ such that $f  = f \circ \phi$.  The automorphism group is only non-trivial on hyperelliptic covers of $\bp^1$ with one point at infinity, in which case it has order 2 and the numerator of \eqref{simhur} is also 2.

\begin{remark}
By the Riemann existence theorem, such a branched cover is prescribed by the location of ramification points and the monodromy around each. Therefore \eqref{factor} and \eqref{simhur} give equivalent definitions.  Any cycle of $T$ acts by conjugation on factorisations \eqref{factor} since it fixes $T$ and preserves the shape of transpositions.  The orbits of this action have size equal to the summands of \eqref{simhur}, which is $\mu_1\cdots\mu_n$, or 1 in the exceptional case $n=1$ and $\mu_1=2$. \end{remark}
\begin{remark}
Different normalisations of simple Hurwitz numbers are often defined in the literature. They may differ by factors of $\mu_1\cdot...\cdot\mu_n$ and $|\text{Aut } \mu|$ where $\text{Aut } \mu$ consists of the permutations of the tuple $\mu = (\mu_1, \mu_2, \ldots, \mu_n)$ that leave it fixed corresponding to whether one distinguishes the preimages of $\infty \in \mathbb{CP}^1$.   For this reason, the cut-and-join relation and ELSV formula will appear here slightly different to some appearances in the literature.  
\end{remark}

Given $f\in\ch_{g,n}(\mu)$, its branching graph \cite{ArnTop} is $f^{-1}(\Gamma_m)\subset\Sigma$ where $\Gamma_m\subset\bc$ is the star graph given by the cone on the $m$th roots of unity. 
\begin{center}
\includegraphics{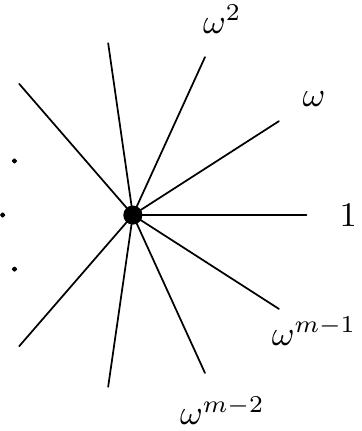}
\end{center}
Conversely, a branching graph, defined below gives rise to $f\in\ch_{g,n}(\mu)$ and the simple Hurwitz numbers $H_{g,n}(\mu)$ equivalently count branching graphs.  

\begin{definition}
We define a {\em branching graph of type $(g; \mu)$} to be an edge-labeled fatgraph of type $(g, \ell(\mu))$ such that for  $m = 2g - 2 + \ell(\mu) + |\mu|$:
\begin{itemize}
\item there are $|\mu|$ vertices and at each of them there is adjacent to $m$ half-edges that are cyclically labeled $1, 2, 3, \ldots, m$;
\item there are exactly $m$ (full) edges that are labeled $1, 2, 3, \ldots, m$; and
\item the $n$ faces are labeled and have perimeters given by $(\mu_1m, \mu_2m, \ldots, \mu_nm)$;
\item each face has a marked $m$-label (of the possible $\mu_k$ appearances of $m$.)
\end{itemize}
The set of all branching graphs of type $(g; \mu)$ is denoted $\text{Fat}_{g,n}(\bm{\mu})$.  
\end{definition}
The marked $m$-labels give locations for removing or attaching edges in the cut-and-join relation below and get rid of any automorphisms.  They also give rise to an unweighted count that produces simple Hurwitz numbers\begin{proposition}
The simple Hurwitz number $H_{g; \mu}$ is equivalent to the enumeration of branching graphs:
\begin{equation*} 
H_{g,n}(\mu_1,...,\mu_n)=\sum_{\Gamma\in\text{Fat}_{g,n}(\mu)}1.
\end{equation*}
\end{proposition}

We now assemble three fundamental results concerning simple Hurwitz numbers.  We use the normalisation
\[
\widehat{H}_{g,n}(\mu_1, \mu_2, \ldots, \mu_n) = \frac{H_{g,n}(\mu_1, \mu_2, \ldots, \mu_n)}{(2g-2+n+|\mu|)!}.
\]

The cut-and-join recursion for simple Hurwitz numbers is obtained by edge removal from branching graphs \cite{GJVNum}:
\begin{align}   \label{caj}
m\hat{H}_{g,n}(\mu_1, \ldots, \mu_n) &= \sum_{i < j} \mu_i\mu_j \hat{H}_{g,n-1}(\bm{\mu}_{S \setminus \{i,j\}}, \mu_i+\mu_j) \\
&+ \frac{1}{2}\sum_{i=1}^n \mu_i\sum_{\alpha+\beta=\mu_i}  \left[ \hat{H}_{g-1,n+1}(\bm{\mu}_{S \setminus\{i\}}, \alpha, \beta) + \mathop{\sum_{g_1+g_2=g}}_{I \sqcup J = S \setminus\{i\}} \hat{H}_{g_1, |I|+1}(\mu_I, \alpha) \hat{H}_{g_2, |J|+1}(\mu_J, \beta) \right].\nonumber
\end{align}
The conditions $g_1+g_2=g$, $I \sqcup J = S \setminus \{i\}$, and $\alpha + \beta = \mu_i$ imply a single edge removal, i.e.\  $m_1 + m_2 = m-1$ where $m=2g-2+n+|\mu|$, $m_1 = 2g_1 - 1 + |I| + |\mu_I| + \alpha$ and $m_2 = 2g_2 - 1 + |J| + |\mu_J| + \beta$. 

\begin{proposition}[ELSV formula, \cite{ELSVHur}] \label{th:elsv}
\[
\widehat{H}_{g,n}(\mu_1, \mu_2, \ldots, \mu_n) =  \prod_{i=1}^n \frac{\mu_i^{\mu_i+1}}{\mu_i!} \int_{\overline{\mathcal{M}}_{g,n}} \frac{1-\lambda_1+...+(-1)^g\lambda_g}{(1-\mu_1\psi_1)...(1-\mu_n\psi_n)}
\]
where $\psi_i$ and $\lambda_i$ are tautological classes on the moduli space.
\end{proposition}
Let $\omega_{g,n}$ denote the correlation differentials output by the Eynard--Orantin topological recursion applied to the spectral curve
\[
x(z) = z \exp(-z) \qquad \text{and} \qquad y(z) = z.
\]  
The Bouchard--Mari\~{n}o conjecture \cite{BMaHur} proven in \cite{BEMSMat,EMSLap} is:
\begin{theorem}  \label{th:bouchardmarino}
The expansion of $\omega_{g,n}$ at $x_1 = x_2 = \cdots = x_n = 0$ is given by
\begin{equation} \label{eq:differentials}
\omega_{g,n} = \sum_{\mu_1, \ldots, \mu_n = 1}^\infty \widehat{H}_{g,n}(\mu_1, \mu_2, \ldots, \mu_n) \prod_{k=1}^n x_k^{\mu_k-1} \, \dd x_k.
\end{equation}
\end{theorem}

\section{Pruned simple Hurwitz numbers}  \label{sec:prusim}

\subsection{Pruned simple Hurwitz numbers}

In the previous section, we interpreted simple Hurwitz numbers as an enumeration of branching graphs. In this section, we define pruned simple Hurwitz numbers by restricting to the set of branching graphs that satisfy a mild condition on the vertex degrees. We will show that simple Hurwitz numbers can be recovered from their pruned counterparts. One advantage of studying pruned simple Hurwitz numbers is that they possess an inherent polynomial structure that allows geometric information to be easily extracted. We conclude the section with an application of this methodology to obtain a new proof of the Witten--Kontsevich theorem.

We define the {\em essential} degree of a vertex in a branching graph to be the number of incident (full) edges. The branching graph of $f\in \ch_{g,n}(\mu)$ can be equivalently described as a triple $(X,\tau_0,\tau_1)$ where $X=f^{-1}(\Gamma_m^0)$, for $\Gamma_m^0$ the interior of the stargraph $\Gamma_m$, equipped with automorphisms $\tau_0:X\to X$ given by the monodromy map around 0 and $\tau_1:X\to X$ given by the monodromy maps around the roots of unity.  The full edges, often simply called edges, correspond to orbits of $\tau_1$ of length 2, whereas half-edges correspond to fixed points of $\tau_1$.    

For $n>0$ and $g\geq 0$ define the set of pruned simple Hurwitz covers:
\begin{align*}
\ck_{g,n}(\mu)=\Big\{f\in\ch_{g,n}(\mu)\mid& \text{ all vertices of the branching graph } f^{-1}(\Gamma_m) \text{ have essential degree }\geq 2.\} 
\end{align*}
We call a branching graph {\em pruned} if all of its vertices have degree at least two and denote the set of all pruned branching graphs of type $(g; \mu)$ by $\text{PFat}_{g,n}(\bm{\mu})$.  Define the pruned simple Hurwitz numbers:
\begin{equation}  \label{prusimhur}
K_{g,n}(\mu_1,...,\mu_n)=\sum_{f\in\ck_{g,n}(\mu)}\frac{\mu_1\cdot...\cdot\mu_n}{|\text{Aut\ }f|}=\sum_{\Gamma\in\text{PFat}_{g,n}(\mu)}1.
\end{equation}
As for $H_{g,n}$, this definition agrees with the definition of pruned simple Hurwitz numbers given in the introduction via factorisations in the symmetric group.  Furthermore, let $m(g, \bm{\mu}) = 2g-2+n+|\bm{\mu}|$ and define the normalisation
\[
\widehat{K}_{g,n}(\bm{\mu}) = \frac{K_{g,n}(\bm{\mu})}{m(g, \bm{\mu})!}.
\]
where $\bm{\mu} = (\mu_1, \mu_2, \ldots, \mu_n)$.

\begin{example} \label{ex:unstable}
The edges of a branching graph with $(g,n) = (0,1)$ necessarily form a tree. So in this case, there does not exist a pruned branching graph and we have $K_{0,1}(\mu_1) = 0$ for all positive integers $\mu_1$.

The edges of a branching graph with $(g,n) = (0,2)$ and $\bm{\mu} = (\mu_1, \mu_2)$ necessarily form a cycle with $\mu_1 + \mu_2$ edges. Remove the edge labeled $\mu_1+\mu_2$ and record the labels of the remaining edges in an anticlockwise fashion around face 1 to obtain a permutation of the set $\{1, 2, \ldots, \mu_1+\mu_2-1\}$. The contribution to the perimeter of face 1 is one more than the number of ascents of the resulting permutation. Therefore, we have
\[
K_{0,2}(\mu_1, \mu_2) = \mu_1 \mu_2A(\mu_1+\mu_2-1, \mu_1-1)
\]
for all positive integers $\mu_1$ and $\mu_2$. Here, $A(m, n)$ represents the Eulerian number that counts the number of permutations of the set $\{1, 2, \ldots, m\}$ with $n$ ascents.
\end{example}

The cut-and-join recursion provides an effective recursive method for the calculation  of simple Hurwitz numbers~\cite{GJaTra}. The next result establishes an analogous recursion for the case of pruned simple Hurwitz numbers.

\begin{proposition}[Cut-and-join recursion for pruned simple Hurwitz numbers] \label{th:cutandjoin}
The following equation holds for all $2g-2+n > 0$ and $\bm{\mu} = (\mu_1, \mu_2, \ldots, \mu_n)$.
\begin{align*}
m(g,\bm{\mu}) \, \widehat{K}_{g,n}(\bm{\mu}) &= \sum_{i < j} \mu_i \mu_j  \hspace{-5mm}\sum_{\alpha + \beta = \mu_i + \mu_j + 1} \hspace{-5mm} \beta \, \widehat{K}_{g,n-1}(\bm{\mu}_{S \setminus \{i,j\}}, \alpha) \\
&+ \frac{1}{2} \sum_{i=1}^n \mu_i \hspace{-2mm}\sum_{\alpha + \beta + \gamma = \mu_i + 1}  \gamma \left[ \widehat{K}_{g-1,n+1}(\bm{\mu}_{S \setminus\{i\}}, \alpha, \beta) + \mathop{\sum_{g_1+g_2=g}}_{I \sqcup J = S \setminus\{i\}}^{\mathrm{stable}} \widehat{K}_{g_1, |I|+1}(\bm{\mu}_I, \alpha) ~ \widehat{K}_{g_2, |J|+1}(\bm{\mu}_J, \beta) \right]
\end{align*}
We use the notation $S = \{1, 2, \ldots, n\}$ and $\bm{\mu}_I = (\mu_{i_1}, \mu_{i_2}, \ldots, \mu_{i_k})$ for $I = \{i_1, i_2, \ldots, i_k\}$. The word {\em stable} over the final summation indicates that summands involving $\widehat{K}_{0,1}$ or $\widehat{K}_{0,2}$ are to be excluded.
\end{proposition}

\begin{example} \label{ex:cutandjoin}
As an example of the cut-and-join recursion for pruned simple Hurwitz numbers in action, consider the following calculation of $\widehat{K}_{0,4}(\mu_1, \mu_2, \mu_3, \mu_4)$, which uses $\widehat{K}_{0,3}(\mu_1, \mu_2, \mu_3) = \mu_1\mu_2\mu_3$.
\begin{align*}
(|\bm{\mu}| + 2) \, \widehat{K}_{0,4}(\mu_1, \mu_2, \mu_3, \mu_4) &= \prod_{i=1}^4\mu_i\cdot\sum_{i < j}\ \sum_{\alpha + \beta = \mu_i + \mu_j + 1} \alpha \beta\\
& =(|\bm{\mu}| + 2)\cdot \prod_{i=1}^4\mu_i\cdot \frac{1}{2}  (\mu_1^2 + \mu_2^2 + \mu_3^2 + \mu_4^2 + \mu_1 + \mu_2 + \mu_3 + \mu_4)
\end{align*}
Therefore we conclude that
$$K_{0,4}(\mu_1, \mu_2, \mu_3, \mu_4) = (|\mu|+2)!\cdot\prod_{i=1}^4\mu_i\cdot\frac{1}{2} \sum_{i=1}^4 (\mu_i^2 + \mu_i)
$$
\end{example}
In contrast, the calculation of $H_{0,4}(\mu_1, \mu_2, \mu_3, \mu_4)$ via the cut-and-join recursion \eqref{caj} is not really feasible because it involves combinatorial identities more difficult than sums of polynomials, and because $H_{0,4}$ appears on both sides of the recursion.

\begin{proof}[Proof of Proposition~\ref{th:cutandjoin}]
We begin by expressing the cut-and-join recursion without the normalisation.
\begin{align*}
K_{g,n}(\bm{\mu}) &= \sum_{i < j} \hspace{3mm}\mu_i  \mu_j \hspace{-6mm} \sum_{\alpha + \beta = \mu_i + \mu_j + 1}  \beta \, \frac{(m-1)!}{(m-\beta)!} \, K_{g,n-1}(\bm{\mu}_{S \setminus \{i,j\}}, \alpha) \\
&+ \frac{1}{2} \sum_{i=1}^n \hspace{3mm}\mu_i  \hspace{-5mm} \sum_{\alpha + \beta + \gamma = \mu_i + 1} \hspace{-5mm}  \gamma \, (m-1)! \left[ \frac{K_{g-1,n+1}(\bm{\mu}_{S \setminus\{i\}}, \alpha, \beta)}{(m-\gamma)!} +\hspace{-2mm} \mathop{\sum_{g_1+g_2=g}}_{I \sqcup J = S \setminus\{i\}}^{\mathrm{stable}} \frac{K_{g_1, |I|+1}(\bm{\mu}_I, \alpha) ~ K_{g_2, |J|+1}(\bm{\mu}_J, \beta)}{m_1! \, m_2!} \right]
\end{align*}
We use the notation $m_1 = 2g_1 - 1 + |I| + |\bm{\mu}_I|+\alpha$ and $m_2 = 2g_2 - 1 + |J| + |\bm{\mu}_J|+\beta$. The conditions $g_1 + g_2 = g$, $I \sqcup J = S \setminus \{i\}$, and $\alpha + \beta + \gamma = \mu_i+1$ imply that $m_1 + m_2 = m - \gamma$.

Recall that $K_{g,n}(\bm{\mu})$ is the number of pruned branching graphs, $\#\text{PFat}_{g,n}(\bm{\mu})$.  Choose a branching graph in $\text{PFat}_{g,n}(\bm{\mu})$ and remove the edge labeled $m$ from it. Repeatedly remove vertices with degree one and their incident edges until all of the vertices of the resulting branching graph have degree at least two. When removing an edge with a given label, we also remove all half-edges with the corresponding label. The removed edges necessarily form a path in the original branching graph. Observe that one of the following three cases must arise.

\begin{itemize}
\item {\em The edge labeled $m$ is adjacent to the face labeled $i$ on both sides and its removal leaves a connected graph.} \\
Suppose that $\gamma$ edges are removed in total, so that a branching graph in $\text{PFat}_{g-1,n+1}(\bm{\mu}_{S \setminus \{i\}}, \alpha, \beta)$ remains, where $\alpha + \beta + \gamma = \mu_i + 1$.

\begin{center}
\includegraphics{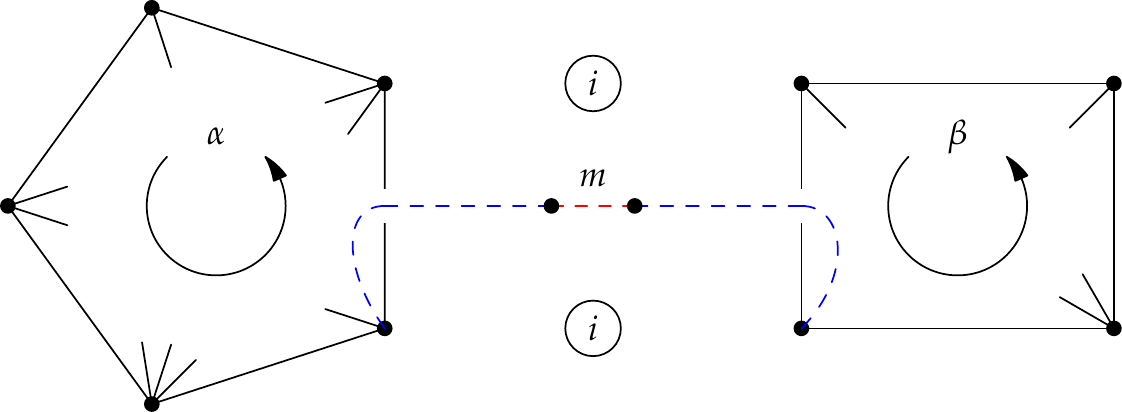}
\end{center}

\noindent Conversely, there are $\frac{1}{2} \, \mu_i \gamma \, \frac{(m-1)!}{(m-\gamma)!}$ ways to reconstruct a branching graph in $\text{PFat}_{g,n}(\bm{\mu})$ from a branching graph in $\text{PFat}_{g-1,n+1}(\bm{\mu}_{S \setminus \{i\}}, \alpha, \beta)$ by adding a path of $\gamma$ edges. When adding an edge with a given label, we also add all possible half-edges with the corresponding label, while maintaining the correct cyclic ordering of the half-edges at every vertex. The factor $\mu_i$ accounts for the position of the new marked $m$-labeled edge. The factor $\gamma$ accounts for the position of the edge labeled $m$ along the path. The factor $\frac{(m-1)!}{(m-\gamma)!}$ accounts for the edge labels appearing on the remaining edges of the path. It is then necessary to adjust by the factor $\frac{1}{2}$ due to the overcounting caused by the symmetry in $\alpha$ and $\beta$.

\item {\em The edge labeled $m$ is adjacent to the face labeled $i$ on both sides and its removal leaves the disjoint union of two connected graphs.} \\
Suppose that $\gamma$ edges are removed in total, so that the disjoint union of two branching graphs in $\text{PFat}_{g_1, |I|+1}(\bm{\mu}_I, \alpha)$ and $\text{PFat}_{g_2, |J|+1}(\bm{\mu}_J, \beta)$ remain, where $\alpha + \beta + \gamma = \mu_i + 1$, $g_1 + g_2 = g$, and $I \sqcup J = S \setminus \{i\}$.

\noindent Conversely, there are $\frac{1}{2} \, \mu_i \gamma \, \frac{(m-1)!}{m_1! \, m_2!}$ ways to reconstruct a branching graph in $\text{PFat}_{g,n}(\bm{\mu})$ from a pair of branching graphs in $\text{PFat}_{g_1, |I|+1}(\bm{\mu}_I, \alpha)$ and $\text{PFat}_{g_2, |J|+1}(\bm{\mu}_J, \beta)$ by adding a path of $\gamma$ edges. When adding an edge with a given label, we also add all possible half-edges with the corresponding label, while maintaining the correct cyclic ordering of the half-edges at every vertex. The factor $\mu_i$ accounts for the position of the new marked $m$-labeled edge. The factor $\gamma$ accounts for the position of the edge labeled $m$ along the path. The factor $\frac{(m-1)!}{m_1! \, m_2!}$ accounts for the distribution of the edge labels $\{1, 2, \ldots, m-1\}$ between the two branching graphs. It is then necessary to adjust by the factor $\frac{1}{2}$ due to the overcounting caused by the symmetry in $(g_1, I, \alpha)$ and $(g_2, J, \beta)$.

\item {\em The edge labeled $m$ is adjacent to two distinct faces labeled $i$ and $j$.} \\
Suppose that $\beta$ edges are removed in total, so that a branching graph in $\text{PFat}_{g,n-1}(\bm{\mu}_{S \setminus \{i, j\}}, \alpha)$ remains, where $\alpha + \beta = \mu_i + \mu_j + 1$.

\begin{center}
\includegraphics{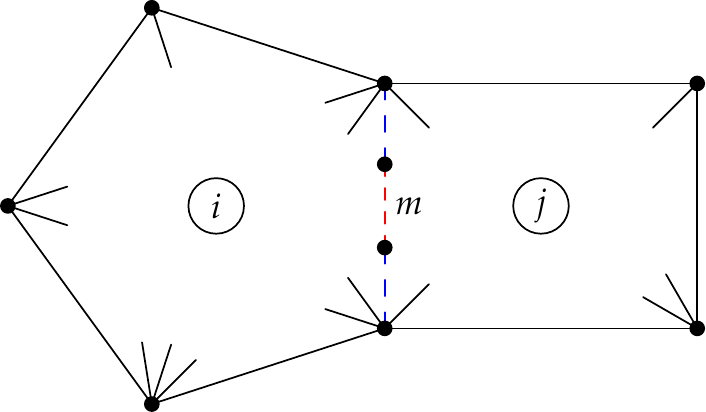}
\end{center}

\noindent Conversely, there are $\mu_i\mu_j \beta \, \frac{(m-1)!}{(m-\beta)!}$ ways to reconstruct a branching graph in $\text{PFat}_{g,n}(\bm{\mu})$ from a branching graph in $\text{PFat}_{g,n-1}(\bm{\mu}_{S \setminus \{i, j\}}, \alpha)$ by adding a path of $\beta$ edges. When adding an edge with a given label, we also add all possible half-edges with the corresponding label, while maintaining the correct cyclic ordering of the half-edges at every vertex. The factor $\mu_i\mu_j$ accounts for the positions of the marked $m$-labeled edges on faces $i$ and $j$.  The factor $\beta$ accounts for the position of the edge labeled $m$ along the path. The factor $\frac{(m-1)!}{(m-\beta)!}$ accounts for the edge labels appearing on the remaining edges of the path.
\end{itemize}

There is a crucial subtlety that arises in the third case, which we now address. One can discern the issue by considering the sequence of diagrams below, in which $\mu_i$ increases from left to right, relative to $\mu_j$.
\begin{center}
\includegraphics{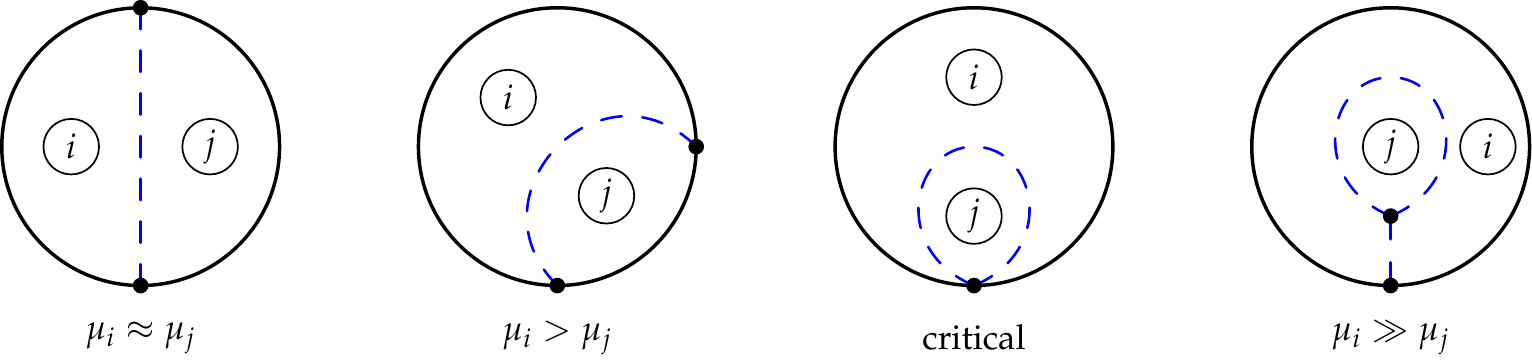}
\end{center}
The factor $\mu_i\mu_j \beta \, \frac{(m-1)!}{(m-\beta)!}$ in the third case actually contributes to diagrams like the one on the far right, in which face $i$ completely surrounds face $j$, or vice versa. In fact, the edge labeled $m$ that we remove can lie anywhere along the dashed path in the schematic diagram. Note that this contributes to the second case, in which the edge labeled $m$ is adjacent to the face labeled $i$ on both sides and its removal leaves the disjoint union of two connected graphs. However, observe that this surplus contribution is precisely equal to the terms from the second case that involve $\widehat{K}_{0,2}$, so one can compensate simply by excluding such terms. Given that we have already witnessed in Example~\ref{ex:unstable} that $\widehat{K}_{0,1} = 0$, we can restrict to the so-called {\em stable} terms in the second case, which are precisely those that do not involve $\widehat{K}_{0,1}$ or $\widehat{K}_{0,2}$.

Therefore, to obtain all fatgraphs in $\text{PFat}_{g,n}(\bm{\mu})$ exactly once, it is necessary to perform the reconstruction process
\begin{itemize}
\item in the first case for all values of $i$ and $\alpha + \beta + \gamma = \mu_i + 1$;
\item in the second case for all {\em stable} values of $i$, $\alpha + \beta + \gamma = \mu_i + 1$, $g_1 + g_2 = g$, and $I \sqcup J = S \setminus \{i\}$; and
\item in the third case for all values of $i$, $j$, and $\alpha + \beta = \mu_i + \mu_j + 1$.
\end{itemize}
We obtain the cut-and-join recursion for pruned simple Hurwitz numbers by summing up over all these contributions.
\end{proof}

\subsection{The pruning correspondence}

Despite the fact that $K_{g,n}(\bm{\mu})$ only counts a subset of the branching graphs enumerated by $H_{g,n}(\bm{\mu})$, simple Hurwitz numbers can be determined from their pruned counterparts, and vice versa. The crucial observation is the following combinatorial result.

\begin{proposition} \label{th:pruning}
The following equation holds for all $(g,n) \neq (0,1)$ and $\bm{\mu} = (\mu_1, \ldots, \mu_n)$.
\[
\widehat{H}_{g,n}(\mu_1, \ldots, \mu_n) = \sum_{\nu_1, \ldots, \nu_n = 1}^{\mu_1, \ldots, \mu_n} \widehat{K}_{g,n}(\nu_1, \ldots, \nu_n) \prod_{i=1}^n \frac{ \mu_i^{\mu_i-\nu_i}}{(\mu_i-\nu_i)!}
\]
\end{proposition}

\begin{proof}
We begin by writing the proposition in the following way.
\[
H_{g,n}(\bm{\mu}) = \sum_{\nu_1, \ldots, \nu_n = 1}^{\mu_1, \ldots, \mu_n} K_{g,n}(\bm{\nu}) \, \frac{(2g-2+n+|\bm{\mu}|)!}{(2g-2+n+|\bm{\nu}|)! \, (\mu_1-\nu_1)! \, \cdots \, (\mu_n-\nu_n)!} \, \prod_{i=1}^n \mu_i^{\mu_i-\nu_i}
\]
This equation encapsulates the fact that, from a branching graph, one obtains a unique pruned branching graph by repeatedly removing vertices with degree one and their incident edges. The process continues until all of the vertices of the resulting branching graph have degree at least two. When removing an edge with a given label, we also remove all half-edges with the corresponding label. It is then necessary to relabel the edges and half-edges in the resulting branching graph so that the new labels come from a set of the form $\{1, 2, \ldots, m\}$, while maintaining the correct cyclic ordering of the half-edges at every vertex. We refer to the process described above as {\em pruning} and observe that it can be carried out one face at a time.

Conversely, every branching graph of type $(g; \bm{\mu})$ can be reconstructed from a pruned branching graph of type $(g; \bm{\nu})$ for $1 \leq \nu_i \leq \mu_i$ by adding $\mu_i - \nu_i$ edges to face $i$, for all $i = 1, 2, \ldots, n$. When adding an edge with a given label, we also add all possible half-edges with the corresponding label. It is then necessary to relabel the edges and half-edges in the resulting branching graph so that the new labels come from a set of the form $\{1, 2, \ldots, m\}$, while maintaining the correct cyclic ordering of the half-edges at every vertex.

There are $K_{g,n}(\bm{\nu})$ possibilities for the pruned branching graph and the factor
\[
\frac{(2g-2+n+|\bm{\mu}|)!}{(2g-2+n+|\bm{\nu}|)! \, (\mu_1-\nu_1)! \, \cdots \, (\mu_n-\nu_n)!}
\]
accounts for the number of ways to choose the set of edge labels for the underlying pruned branching graph as well as the set of $\mu_i - \nu_i$ edge labels to be added to face $i$ for $i = 1, 2, \ldots, n$. 

All that remains is to show that the factor $\mu^{\mu - \nu}$  is equal to the number of ways to add $\mu - \nu$ edges to a pruned face with perimeter $\nu$. To do this, we invoke the following generalisation of Cayley's formula.
\begin{quote}
Let $N \subseteq M$ be sets of size $\nu \leq \mu$, respectively. Then the number of rooted forests on $\mu$ vertices labeled by $M$ with $\nu$ components whose roots are labeled by $N$ is precisely $T(\mu, \nu) = \nu \mu^{\mu - \nu - 1}$.
\end{quote}
See for example \cite{AZiPro} for a proof of the formula for $T(\mu, \nu)$.

Consider a face of perimeter $\mu$ in a branching graph that has perimeter $\nu$ after pruning. By the definition of a branching graph, each edge label occurs precisely $\nu$ times in the pruned face, so we can divide its perimeter into $\nu$ disjoint {\em intervals}, each of which contains all of the edge labels. From the unpruned face of perimeter $\mu$, construct a rooted forest by contracting each of the intervals to a root vertex and reassign each edge label to the adjacent vertex that is further away from the root. We thus obtain a rooted forest with $\nu$ components, $\mu-\nu$ edges, and hence $\mu$ vertices. The $\nu$ roots are labeled by their corresponding intervals, while the remaining $\mu-\nu$ vertices are labeled by distinct positive integers derived from the original edge labels.

As an example, consider the diagram below left, which shows a pruned face of perimeter $\nu = 3$ with $\mu-\nu=8$ edges added to create a face of perimeter $\mu = 11$. The corresponding rooted forest is shown below right.
\begin{center}
\includegraphics{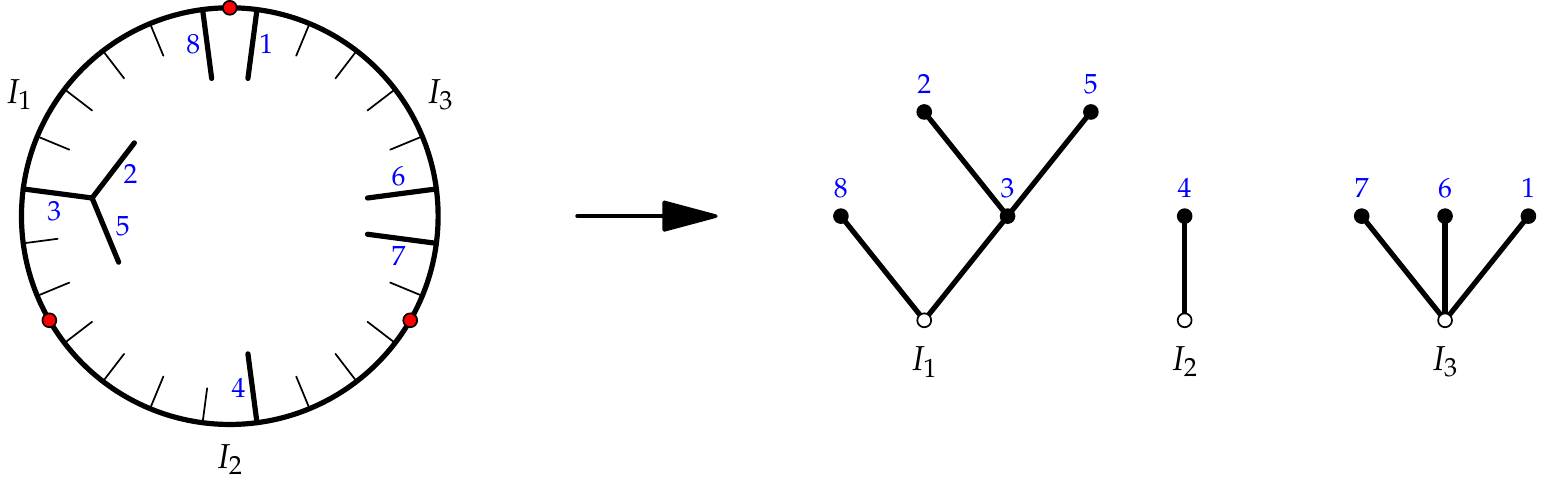}
\end{center}

So there are $\frac{\mu}{\nu}T(\mu, \nu) = \mu^{\mu-\nu}$ possibilities for the resulting rooted forest. Conversely, the process may be reversed to construct a face of perimeter $\mu$ from a pruned face of perimeter $\nu$ together with a labeled rooted forest with $\nu$ components and $\mu$ vertices. The edge labels determine the cyclic orientations of the edges adjacent to a given vertex.
\end{proof}

Note that the system of linear equations in Proposition~\ref{th:pruning} relating the values of $\widehat{H}_{g,n}$ to those of $\widehat{K}_{g,n}$ is triangular in the sense that $\widehat{H}_{g,n}(\bm{\mu})$ depends only on values of $\widehat{K}_{g,n}(\bm{\nu})$ for which $\bm{\nu} \leq \bm{\mu}$ in the lexicographical order. Therefore, all of the information stored in the simple Hurwitz numbers is theoretically also stored in their pruned counterparts.

Theorem~\ref{th:bouchardmarino} states that the simple Hurwitz numbers comprise a natural enumerative problem in the context of the Eynard--Orantin topological recursion. The following result demonstrates that the same is true of the pruned simple Hurwitz numbers and, furthermore, that they can be derived from the same spectral curve.

\begin{proposition} \label{pro:expansion}
For $2g-2+n>0$, the expansions of the simple Hurwitz differentials of equation~(\ref{eq:differentials}) at the point $z_1 = z_2 = \cdots = z_n = 0$ satisfy
\[
\omega_{g,n} = \sum_{\mu_1, \ldots, \mu_n = 1}^\infty \widehat{K}_{g,n}(\mu_1, \ldots, \mu_n) \prod_{i=1}^n z_i^{\mu_i-1} \, \dd z_i.
\]
\end{proposition}

\begin{proof}
Recall that the simple Hurwitz differentials are defined in equation~(\ref{eq:differentials}) by the formula
\[
\omega_{g,n} = \sum_{\mu_1, \ldots, \mu_n = 1}^\infty \widehat{H}_{g,n}(\mu_1, \ldots, \mu_n) \prod_{i=1}^n x_i^{\mu_i-1} \, \dd x_i.
\]
Furthermore, recall that $\omega_{g,n}$ is a meromorphic multidifferential on $C^n$, where $C$ is the rational spectral curve given parametrically by the equation $x(z) = z \exp(-z)$ and $y(z) = z$. We let $z_1, z_2, \ldots, z_n$ be the rational coordinates on $C^n$ and define $x_1 = x(z_1), x_2 = x(z_2), \ldots, x_n = x(z_n)$.

Now define another family of multidifferentials $\overline{\omega}_{g,n}$ on $C^n$ by the following local expansion at the point $z_1 = z_2 = \cdots = z_n = 0$.\footnote{Proposition~\ref{th:polynomial} below asserts that $\widehat{K}_{g,n}$ is a polynomial, so the equation does indeed define an analytic multidifferential.}
\[
\overline{\omega}_{g,n} = \sum_{\mu_1, \ldots, \mu_n = 1}^\infty \widehat{K}_{g,n}(\mu_1, \ldots, \mu_n) \prod_{i=1}^n z_i^{\mu_i-1} \dd z_i
\]

Of course, we wish to prove that $\overline{\omega}_{g,n} = \omega_{g,n}$ and we proceed by calculating the following residue.
\begin{align*}
\mathop{\text{Res}}_{x_1=0} \, \cdots \, \mathop{\text{Res}}_{x_n=0} \overline{\omega}_{g,n} \prod_{i=1}^n x_i^{-\mu_i} &= \mathop{\text{Res}}_{z_1=0} \, \cdots \, \mathop{\text{Res}}_{z_n=0} \sum_{\nu_1, \ldots, \nu_n = 1}^\infty \widehat{K}_{g,n}(\nu_1, \ldots, \nu_n) \prod_{i=1}^n  z_i^{\nu_i-1} \, \dd z_i \left[ z_i \exp(-z_i) \right]^{-\mu_i} \\
&= \mathop{\text{Res}}_{z_1=0} \, \cdots \, \mathop{\text{Res}}_{z_n=0} \sum_{\nu_1, \ldots, \nu_n = 1}^\infty \widehat{K}_{g,n}(\nu_1, \ldots, \nu_n) \prod_{i=1}^n z_i^{\nu_i-\mu_i-1} \, \dd z_i \, \sum_{m_i=0}^\infty \frac{\mu_i^{m_i}}{m_i!} z_i^{m_i} \\
&= \sum_{\nu_1, \ldots, \nu_n = 1}^{\mu_1, \ldots, \mu_n} \widehat{K}_{g,n}(\nu_1, \ldots, \nu_n) \prod_{i=1}^n \mathop{\text{Res}}_{z_i=0} z_i^{\nu_i-\mu_i-1} \, \dd z_i \, \sum_{m_i=0}^\infty \frac{\mu_i^{m_i}}{m_i!} z_i^{m_i} \\
&= \sum_{\nu_1, \ldots, \nu_n = 1}^{\mu_1, \ldots, \mu_n} \widehat{K}_{g,n}(\nu_1, \ldots, \nu_n) \prod_{i=1}^n \frac{\mu_i^{\mu_i-\nu_i}}{(\mu_i-\nu_i)!} \\
&= \widehat{H}_{g,n}(\mu_1, \ldots, \mu_n) 
\end{align*}
The last equality here is a direct consequence of Proposition~\ref{th:pruning}. It now follows from the above residue calculation that
\[
\overline{\omega}_{g,n} = \sum_{\mu_1, \ldots, \mu_n = 1}^\infty \widehat{H}_{g,n}(\mu_1, \ldots, \mu_n) \prod_{i=1}^n  x_i^{\mu_i-1} \, \dd x_i = \omega_{g,n}. \qedhere
\]
\end{proof}

In Example~\ref{th:polynomial}, we observed that the linear factor $m(g, \bm{\mu}) = 2g-2+n+|\bm{\mu}|$ on the left hand side of the cut-and-join recursion divides the right hand side in the case $(g,n) = (0,4)$, thereby establishing the fact that $\widehat{K}_{0,4}$ is a polynomial. In fact, we will see that this phenomenon is general.

\begin{lemma}
For non-negative integers $d$, define the sequence $q_d(1), q_d(2), q_d(3), \ldots$ by the triangular system of linear equations
\[
\frac{\mu^{\mu+d+1}}{\mu!} = \sum_{\nu=1}^\mu q_d(\nu) \, \frac{\nu \mu^{\mu-\nu}}{(\mu-\nu)!}, \qquad \text{for } \mu = 1, 2, 3, \ldots.
\]
Then $q_d$ is a polynomial of degree $2d$.
\end{lemma}

\begin{proof}
First, observe that $q_0(\nu) = 1$ for $\nu = 1, 2, 3, \ldots$, since this fact is equivalent to the identity
\[
\mu^2 T(\mu, 1) = \sum_{\nu=1}^\mu \frac{\mu!}{(\mu-\nu)!} \, T(\mu, \nu), \qquad \text{for } \mu = 1, 2, 3, \ldots.
\]
As in the proof of Proposition~\ref{th:pruning}, $T(\mu, \nu) = \nu \mu^{\mu-\nu-1}$ denotes the number of rooted forests on $\mu$ labeled vertices with $\nu$ labeled roots. We interpret the left hand side of this equation as the number of trees with vertices labeled $1, 2, \ldots, \mu$, along with a choice of an initial vertex and a terminal vertex, which are allowed to coincide. Given such a tree, suppose that there are $\nu$ vertices on the unique path from the initial vertex to the terminal vertex. Note that $1 \leq \nu \leq \mu$ and that there are $\frac{\mu!}{(\mu-\nu)!}$ possibilities for the labels of the vertices along the path. Removing the edges on the path yields a rooted forest, whose roots are precisely those vertices on the path. The number of such rooted forests is $T(\mu, \nu)$ by definition, which leads to the expression on the right hand side of this equation.

Second, consider the following sequence of equalities.
\begin{align*}
\sum_{\nu=1}^\mu q_{d+1}(\nu) \, \frac{\nu \mu^{\mu-\nu}}{(\mu-\nu)!} =& \mu \sum_{\nu=1}^\mu q_d(\nu) \, \frac{\nu \mu^{\mu-\nu}}{(\mu-\nu)!} \\
=& \sum_{\nu=1}^\mu \nu q_d(\nu) \, \frac{\nu \mu^{\mu-\nu}}{(\mu-\nu)!} + \sum_{\nu=1}^\mu (\mu-\nu) q_d(\nu) \, \frac{\nu \mu^{\mu-\nu}}{(\mu-\nu)!} \\
=& \sum_{\nu=1}^\mu \nu q_d(\nu) \, \frac{\nu \mu^{\mu-\nu}}{(\mu-\nu)!} + \mu \sum_{\nu=1}^\mu (\nu-1) q_d(\nu-1) \, \frac{\mu^{\mu-\nu}}{(\mu-\nu)!} \\
=& \sum_{\nu=1}^\mu \nu q_d(\nu) \, \frac{\nu \mu^{\mu-\nu}}{(\mu-\nu)!} + \sum_{\nu=1}^\mu (\nu-1) q_d(\nu-1) \, \frac{\nu \mu^{\mu-\nu}}{(\mu-\nu)!} + \mu \sum_{\nu=1}^\mu (\nu-2) q_d(\nu-2) \, \frac{\mu^{\mu-\nu}}{(\mu-\nu)!} \\
=& \sum_{\nu=1}^\mu \nu q_d(\nu) \, \frac{\nu \mu^{\mu-\nu}}{(\mu-\nu)!} + \sum_{\nu=1}^\mu (\nu-1) q_d(\nu-1) \, \frac{\nu \mu^{\mu-\nu}}{(\mu-\nu)!} + \cdots + \sum_{\nu=1}^\mu 1 q_d(1) \, \frac{\nu \mu^{\mu-\nu}}{(\mu-\nu)!} \\
=& \sum_{\nu=1}^\mu \left[ \nu q_d(\nu) + (\nu-1) q_d(\nu-1) + \cdots + 1 q_d(1) \right] \, \frac{\nu \mu^{\mu-\nu}}{(\mu-\nu)!}
\end{align*}
Since the sequences $q_d(1), q_d(2), q_d(3), \ldots$ are defined by triangular systems of linear equations, we may deduce from the above sequence of equalities that
\begin{equation} \label{eq:qpolynomials}
q_{d+1}(\nu) = \sum_{i=1}^\nu i q_d(i) \qquad \text{and} \qquad q_{d+1}(\nu) = q_{d+1}(\nu-1) + v q_d(\nu).
\end{equation}
Using the base case $q_0(\nu) = 1$ and equation~(\ref{eq:qpolynomials}), it is now straightforward to prove by induction that $q_d$ is a polynomial of degree $2d$.
\end{proof}

\begin{proposition} \label{th:polynomial}
For $2g-2+n > 0$, the normalised pruned simple Hurwitz number $\widehat{K}_{g,n}(\mu_1, \mu_2,  \ldots, \mu_n)$ is a polynomial in $\mu_1, \mu_2, \ldots, \mu_n$ of degree $6g-6+3n$.
\end{proposition}

\begin{proof}
Substitute the ELSV formula --- see Proposition~\ref{th:elsv} --- into the equation in the statement of Proposition~\ref{th:pruning}:
\[
\widehat{H}_{g,n}(\mu_1, \ldots, \mu_n) = \sum_{\nu_1, \ldots, \nu_n}^{\mu_1, \ldots, \mu_n} \widehat{K}_{g,n}(\nu_1, \ldots, \nu_n) \prod_{i=1}^n  \frac{ \mu_i^{\mu_i-\nu_i}}{(\mu_i-\nu_i)!}
\]
to obtain the following:
\[
\sum_{|\mathbf{d}|+\ell = 3g-3+n} (-1)^\ell \langle \tau_{d_1} \cdots \tau_{d_n} \lambda_\ell \rangle_g \, \prod_{i=1}^n \frac{\mu_i^{\mu_i+d_i+1}}{\mu_i !} = \sum_{\nu_1, \ldots, \nu_n}^{\mu_1, \ldots, \mu_n} \widehat{K}_{g,n}(\nu_1, \ldots, \nu_n) \prod_{i=1}^n \frac{ \mu_i^{\mu_i-\nu_i}}{(\mu_i-\nu_i)!}.
\]

From the definition of the polynomials $q_d$ for $d = 0, 1, 2, \ldots$, we may deduce from this equation that for all positive integers $\nu_1, \nu_2, \ldots, \nu_n$,
\begin{equation} \label{eq:prunedelsv}
\widehat{K}_{g,n}(\nu_1, \ldots, \nu_n) = \prod_{i=1}^n\nu_i\cdot\hspace{-3mm}\sum_{|\mathbf{d}|+\ell = 3g-3+n} (-1)^\ell \langle \tau_{d_1} \cdots \tau_{d_n} \lambda_\ell \rangle_g \, \prod_{i=1}^n q_{d_i}(\nu_i).
\end{equation}
Since we have already shown that $q_d$ is a polynomial of degree $2d$, the desired result follows.
\end{proof}

The sequence of polynomials $q_0, q_1, q_2, \ldots$ plays a crucial part in the relation between simple Hurwitz numbers and their pruned counterparts. The numerators of the corresponding triangle of coefficients appear as sequence A202339 in the the On-Line Encyclopedia of Integer Sequences~\cite{oeis}. We state without proof some facts about this sequence of polynomials, which follow from the base case $q_0 = 1$ and the recursive definition in equation~(\ref{eq:qpolynomials}).

\begin{proposition} \label{pro:qpolynomials}
The function $q_d$ is in fact a polynomial of degree $2d$ with leading coefficient  $a_d = \frac{1}{(2d)!!}$. For all non-negative integers $d$ and positive integers $\nu$, $q_d(\nu) = S(\nu+d,\nu)$, the Stirling number of the second kind that represents the number of ways to partition a set with $\nu+d$ objects into $\nu$ non-empty subsets. 
\end{proposition}

The combinatorial significance of the Stirling numbers of the second kind for pruned simple Hurwitz numbers is presently unclear.

\begin{center}
\begin{tabular}{cl} \toprule
$d$ & $q_d(\nu)$ \\ \midrule
0 & 1 \\
1 & $\frac{1}{2} (\nu^2 + \nu)$ \\
2 & $\frac{1}{24} (3\nu^4 + 10\nu^3 + 9\nu^2 + 2\nu)$ \\
3 & $\frac{1}{48} (\nu^6 + 7\nu^5 + 17\nu^4 + 17\nu^3 + 6\nu^2)$ \\
4 & $\frac{1}{5760} (15\nu^8 + 180\nu^7 + 830\nu^6 + 1848\nu^5 + 2015\nu^4 + 900\nu^3 + 20\nu^2 - 48\nu)$ \\
5 & $\frac{1}{11520} (3\nu^{10} + 55\nu^9 + 410\nu^8 + 1598\nu^7 + 3467\nu^6 + 4055\nu^5 + 2120\nu^4 + 52\nu^3 - 240\nu^2)$ \\ \bottomrule
\end{tabular}
\end{center}

\subsection{Witten--Kontsevich theorem}

We apply the earlier results of this section to give a direct proof of the Witten--Kontsevich theorem, which governs intersection numbers of psi-classes on Deligne--Mumford moduli spaces of curves $\overline{\mathcal M}_{g,n}$. We adopt the following notation of Witten for such intersection numbers, which are defined to be zero unless the condition $|\bm{d}| = \dim_\mathbb{C} \overline{\mathcal M}_{g,n} = 3g-3+n$ is satisfied. 
\[
\langle \tau_{d_1} \cdots \tau_{d_n} \rangle_g = \int_{\overline{\mathcal M}_{g,n}} \psi_1^{d_1} \cdots \psi_n^{d_n}
\]
The psi-classes $\psi_1, \psi_2, \ldots, \psi_n \in H^2(\overline{\mathcal M}_{g,n}; \mathbb{Q})$ are the first Chern classes of the cotangent line bundles at the marked points. For more information on Deligne--Mumford moduli spaces of curves, psi-classes, and the Witten--Kontsevich theorem, see the book of Harris and Morrison~\cite{HMoMod}.

One of the virtues of the cut-and-join recursion for pruned simple Hurwitz numbers is that, although it is primarily an equality of numbers, it can be interpreted as an equality of polynomials in light of Proposition~\ref{th:polynomial}. In order to do this, we define the following functions for non-negative integers $i$ and $j$.
\[
P_i(x, y) = \sum_{\alpha + \beta = x + y + 1} \alpha \beta \, q_i(\alpha) \qquad \text{and} \qquad P_{i,j}(x) = \sum_{\alpha + \beta + \gamma = x + 1} \alpha \beta \gamma \, q_i(\alpha) q_j(\beta)
\]
The following lemma will be useful to determine the leading order behaviour of $P_i$ and $P_{i,j}$.

\begin{lemma}
The expression
\[
\sum_{\alpha_1 + \cdots + \alpha_m = n} \alpha_1^{k_1} \cdots \alpha_m^{k_m}
\]
is a polynomial in $n$ of degree $|\bm{k}| + m - 1$ with leading coefficient $\frac{k_1! \cdots k_m!}{(|\bm{k}| + m - 1)!}$.
\end{lemma}

One proof of this fact expresses monomials $\alpha^k$ as linear combinations of binomial coefficients $\alpha^k = k! \binom{\alpha}{k} + \cdots$ and uses the combinatorial fact
\[
\sum_{\alpha_1 + \cdots + \alpha_m = n} \binom{\alpha_1}{k_1} \cdots \binom{\alpha_m}{k_m} = \binom{n+m-1}{|\bm{k}|+m-1}.
\]
As a direct consequence of this lemma and Proposition~\ref{pro:qpolynomials}, we have the following result.

\begin{corollary} \label{cor:leading}
For non-negative integers $i$ and $j$, $P_i$ is a polynomial of degree $2i+3$ and $P_{i,j}$ is a polynomial of degree $2i+2j+5$. Their leading coefficients are given by the formulae
\[
\left[ x^{2a+1} y^{2b} \right] P_{a+b-1}(x,y) = \frac{(2a+2b-1)!!}{(2a+1)! \, (2b)!} \qquad \text{and} \qquad \left[ x^{2a+2b+5} \right] P_{a,b}(x) = \frac{(2a+1)!! (2b+1)!!}{(2a+2b+5)!}.
\]
\end{corollary}

\begin{example}
The summations over $\alpha, \beta, \gamma$ in the cut-and-join recursion for pruned simple Hurwitz numbers can be replaced by expressions involving $P_i$ and $P_{i,j}$. For example, consider the case $(g,n) = (1,2)$.
\begin{align*}
& (\mu_1 + \mu_2 + 2) \, \widehat{K}_{1,2}(\mu_1, \mu_2) \\
=& \sum_{\alpha + \beta = \mu_1 + \mu_2 + 1} \alpha \beta \, \widehat{K}_{1,1}(\alpha) + \frac{1}{2} \sum_{\alpha + \beta + \gamma = \mu_1 + 1} \alpha \beta \gamma \, \widehat{K}_{0,3}(\mu_2, \alpha, \beta) + \frac{1}{2} \sum_{\alpha + \beta + \gamma = \mu_2 + 1} \alpha \beta \gamma \, \widehat{K}_{0,3}(\mu_1, \alpha, \beta) \\
=& \sum_{\alpha + \beta = \mu_1 + \mu_2 + 1} \alpha \beta \, \frac{q_1(\alpha) - q_0(\alpha)}{24} + \frac{1}{2} \sum_{\alpha + \beta + \gamma = \mu_1 + 1} \alpha \beta \gamma \, q_0(\mu_2) q_0(\alpha) q_0(\beta) + \frac{1}{2} \sum_{\alpha + \beta + \gamma = \mu_2 + 1} \alpha \beta \gamma \, q_0(\mu_1) q_0(\alpha) q_0(\beta) \\
=& \frac{1}{24} P_1(\mu_1, \mu_2) - \frac{1}{24} P_0(\mu_1, \mu_2) + \frac{1}{2} P_{0,0}(\mu_1) \, q_0(\mu_2) + \frac{1}{2} P_{0,0}(\mu_2) \, q_0(\mu_1)
\end{align*}
\end{example}

We are now in a position to deduce the Witten--Kontsevich theorem from equation~(\ref{eq:prunedelsv}) and Proposition~~\ref{eq:prunedelsv}, the cut-and-join recursion for pruned simple Hurwitz numbers.

\begin{theorem}[Witten--Kontsevich theorem] \label{thm:witten}
The intersection numbers of psi-classes on the Deligne--Mumford moduli spaces of curves $\overline{\mathcal M}_{g,n}$ satisfy the following equation for all $d_1, d_2, \ldots, d_n$.
\begin{align*}
\langle \tau_{d_1} \cdots \tau_{d_n} \rangle &= \sum_{j=2}^n \frac{(2d_1+2d_j-1)!!}{(2d_1+1)!! \, (2d_j-1)!!} \, \langle \tau_{\bm{d}_{S \setminus \{1, j\}}} \tau_{d_1+d_j-1} \rangle \\
&+ \frac{1}{2} \sum_{i+j=d_1-2} \frac{(2i+1)!! \, (2j+1)!!}{(2d_1+1)!!} \left[ \langle \tau_i \tau_j \tau_{\bm{d}_{S \setminus \{1\}}} \rangle + \sum_{I \sqcup J = S \setminus \{1\}} \langle \tau_i \tau_{\bm{d}_I} \rangle \, \langle \tau_j \tau_{\bm{d}_J} \rangle \right]
\end{align*}
\end{theorem}


\begin{remark}
In actual fact, the original formulation of Witten posited that a certain natural generating function for intersection numbers of psi-classes --- the Gromov--Witten potential of a point --- is a solution to the KdV integrable hierarchy~\cite{WitTwo}. This is equivalent to the fact that the generating function is annihilated by the Virasoro differential operators $L_{-1}, L_0, L_1, \ldots$, which satisfy the Virasoro relation $[L_m, L_n] = (m-n) L_{m+n}$. The annihilation by $L_{-1}$ and $L_0$ is equivalent to the dilaton and string equations, which have straightforward geometric interpretations that already appear in the original paper of Witten. It is straightforward to prove that Theorem~\ref{thm:witten} is equivalent to the fact that $L_{d_1-1}$ annihilates the Gromov--Witten potential of a point.
\end{remark}

\begin{proof}[Proof of Theorem~\ref{thm:witten}]
Take the cut-and-join recursion for pruned simple Hurwitz numbers and consider the coefficient of $\mu_1 \bm{\mu}^{2\bm{d}} = \mu_1^{2d_1+1} \mu_2^{2d_2} \cdots \mu_n^{2d_n}$ for $|\bm{d}| = 3g-3+n$. This condition ensures that no terms involving non-trivial Hodge classes appear.

The desired coefficient of the left hand side of the cut-and-join recursion can be expressed as follows.
\begin{align*}
& \left[ \mu_1 \bm{\mu}^{2\bm{d}} \right] (2g-2+n+|\bm{\mu}|) \, \widehat{K}_{g,n}(\bm{\mu}_S) \\
=& \left[ \mu_1 \bm{\mu}^{2\bm{d}} \right] (2g-2+n+|\bm{\mu}|) \sum_{|\mathbf{k}|+\ell = 3g-3+n} (-1)^\ell \langle \tau_{k_1} \cdots \tau_{k_n} \lambda_\ell \rangle_g \, \prod_{i=1}^n q_{k_i}(\mu_i) \\
=& \langle \tau_{d_1} \cdots \tau_{d_n} \rangle_g \, \prod_{i=1}^n a_{d_i} \\
\end{align*}
The first equality uses equation~(\ref{eq:prunedelsv}) while the second makes use of the fact that $q_d$ is a polynomial of degree $2d$ with leading coefficient $a_d = \frac{1}{(2d)!!}$, as stated in Proposition~\ref{pro:qpolynomials}.

The desired coefficient of the first term on the right hand side of the cut-and-join recursion can be expressed as follows.
\begin{align*}
& \left[ \mu_1 \bm{\mu}^{2\bm{d}} \right] \sum_{i < j} \sum_{\alpha + \beta = \mu_i + \mu_j + 1} \alpha \beta \, \widehat{K}_{g,n-1}(\bm{\mu}_{S \setminus \{i,j\}}, \alpha) \\
=& \left[ \mu_1 \bm{\mu}^{2\bm{d}} \right] \sum_{i < j} \sum_{|\bm{k}_{S \setminus\{i,j\}}|+s+\ell=3g-4+n} (-1)^\ell \langle \tau_{\bm{k}_{S \setminus \{i,j\}}} \tau_s \lambda_\ell \rangle_g \, P_s(\mu_i, \mu_j) \prod_{m \in S \setminus \{i,j\}} q_{k_m}(\mu_m) \\
=& \sum_{j=2}^n \langle \tau_{\bm{d}_{S \setminus \{i,j\}}} \tau_{d_1+d_j-1} \rangle_g \left[ \mu_1^{2d_1+1} \mu_j^{2d_j} \right] P_{d_1+d_j-1}(\mu_1, \mu_j) \prod_{m \in S \setminus \{1,j\}} a_{d_m} \\
=& \sum_{j=2}^n \langle \tau_{\bm{d}_{S \setminus \{i,j\}}} \tau_{d_1+d_j-1} \rangle_g \, \frac{(2d_1+2d_j-1)!!}{(2d_1+1)! (2d_j)!} \prod_{m \in S \setminus \{1,j\}} a_{d_m}
\end{align*}
The first equality uses equation~(\ref{eq:prunedelsv}), the second takes into account the fact that $|\bm{d}| = 3g-3+n$, while the third follows from Corollary~\ref{cor:leading}.

In an analogous fashion, the desired coefficients of the second and third terms on the right hand side of the cut-and-join recursion can be expressed as follows.
\[
\frac{1}{2} \sum_{s+t=d_1-2}  \langle \tau_{\bm{d}_{S \setminus \{1\}}} \tau_s \tau_t \rangle_{g-1} \left[ \mu_1^{2d_1+1} \right] P_{s,t}(\mu_1) \prod_{m \in S \setminus \{1\}} a_{d_m}
\]
\[
\frac{1}{2} \mathop{\sum_{g_1+g_2=g}}_{I \sqcup J = S \setminus\{1\}}^{\mathrm{stable}} \sum_{s+t=d_1-2} \langle \tau_{\bm{d}_I} \tau_s \rangle_{g_1} \langle \tau_{\bm{d}_J} \tau_t \rangle_{g_2} \frac{(2s+1)!! (2t+1)!!}{(2d_1+1)!} \prod_{m \in S \setminus \{1\}} a_{d_m}
\]


Now substitute these expressions into the cut-and-join recursion and divide both sides by $a_{d_1} a_{d_2} \cdots a_{d_n}$ to obtain the desired result.
\end{proof}

It is worth remarking that Okounkov and Pandharipande also deduce the Witten--Kontsevich theorem using the ELSV formula as a starting point~\cite{OPaGrom}. Their approach expresses the asymptotics of simple Hurwitz numbers as a sum over trivalent ribbon graphs, thereby obtaining Kontsevich's combinatorial formula~\cite{KonInt}. The Witten--Kontsevich theorem is then derived as a consequence of this formula using the theory of matrix models. In contrast, the notion of pruning reduces the enumeration of simple Hurwitz numbers to an equivalent problem that is inherently polynomial. The asymptotic analysis of pruned simple Hurwitz numbers is then stored in the top degree terms of the cut-and-join recursion. As shown in the proof of Theorem~\ref{thm:witten} above, the Witten--Kontsevich theorem emerges directly from this analysis without necessitating the use of a matrix model.


There are now myriad proofs of the Witten--Kontsevich theorem, most of which involve the theory of matrix models in one way or another. Exceptional in this respect is the proof by Mirzakhani, who analyses the volume $V_{g,n}(L_1, L_2, \ldots, L_n)$ of the moduli space of genus $g$ hyperbolic surfaces with $n$ geodesic boundary components of lengths $L_1, L_2, \ldots, L_n$~\cite{MirWei}. Her proof consists of two parts --- a theorem that relates $V_{g,n}(L_1, L_2, \ldots, L_n)$ to the intersection theory of moduli spaces of curves and a recursion that can be used to compute $V_{g,n}(L_1, L_2, \ldots, L_n)$. It is natural to consider these as analogous to the ELSV formula and the cut-and-join recursion, respectively. Our proof of the Witten--Kontsevich theorem bears strong resemblance to that of Mirzakhani, but uses a combinatorial argument rather than hyperbolic geometry to obtain the recursion.

We finish the section with a table of the polynomials $\widehat{K}_{g,n}(\mu_1, \mu_2, \ldots, \mu_n)$ which give pruned simple Hurwitz numbers.
\begin{center}
\begin{tabular}{ccl} \toprule
$g$ & $n$ & $\widehat{K}_{g,n}(\mu_1, \mu_2, \ldots, \mu_n)$ \\ \midrule
0 & 3 & 1 \\
0 & 4 & $\frac{1}{2} \sum \mu_i^2 + \frac{1}{2} \sum \mu_i$ \\
0 & 5 & $\frac{1}{8} \sum \mu_i^4 + \frac{1}{2}  \sum \mu_i^2 \mu_j^2 + \frac{5}{12} \sum \mu_i^3 + \frac{1}{2} \sum \mu_i^2 \mu_j + \frac{3}{8} \sum \mu_i^2 + \frac{1}{2} \sum \mu_i \mu_j + \frac{1}{12} \sum \mu_i$ \\
1 & 1 & $\frac{1}{48} \mu_1^2 + \frac{1}{48} \mu_1 - \frac{1}{24}$ \\
1 & 2 & $\frac{1}{192} (\mu_1^4 + \mu_2^4) + \frac{1}{96} \mu_1^2 \mu_2^2 + \frac{5}{288} (\mu_1^3 + \mu_2^3) + \frac{1}{96} (\mu_1^2 \mu_2 + \mu_1 \mu_2^2) - \frac{1}{192} (\mu_1^2 + \mu_2^2) + \frac{1}{96} \mu_1 \mu_2 - \frac{5}{288} (\mu_1 + \mu_2)$ \\
2 & 1 & $\frac{1}{442368} \mu_1^8 + \frac{1}{36864} \mu_1^7 + \frac{271}{3317760} \mu_1^6 - \frac{7}{276480} \mu_1^5 - \frac{1873}{6635520} \mu_1^4 - \frac{53}{552960} \mu_1^3 + \frac{329}{1658880} \mu_1^2 + \frac{13}{138240}\mu_1$ \\ \bottomrule
\end{tabular}
\end{center}

\section{Pruned orbifold Hurwitz numbers}  \label{sec:pruorb}

\subsection{Orbifold Hurwitz numbers}

In this section, we generalise the results for simple Hurwitz numbers in the previous section to the case of orbifold Hurwitz numbers.

\begin{definition}
For a fixed positive integer $a$, the {\em orbifold Hurwitz number} $H_{g,n}^{[a]}(\mu_1, \mu_2, \ldots, \mu_n)$ is the weighted enumeration of connected genus $g$ branched covers $f: (\Sigma; p_1, p_2, \ldots, p_n) \to (\mathbb{CP}^1; \infty)$ such that
\begin{itemize}
\item the preimage of $\infty$ is given by the divisor $\mu_1 p_1 + \mu_2 p_2 + \cdots + \mu_n p_n$;
\item the ramification profile over 0 is given by a partition of the form $(a, a, \ldots, a)$; and
\item the only other ramification is simple and occurs over $m$ fixed points.
\end{itemize}
\end{definition}

Note that we recover the definition of simple Hurwitz numbers in the case $a = 1$. Justification for the terminology {\em orbifold Hurwitz number} stems from the following generalisation of the ELSV formula due to Johnson, Pandharipande, and Tseng.

\begin{theorem}[Orbifold ELSV formula~\cite{JPTAbe}]
\[
H_{g,n}^{[a]}(\mu_1, \mu_2, \ldots, \mu_n) = \prod_{i=1}^n\mu_i\cdot\left( 2g-2+n+\frac{|\mu|}{a} \right)! a^{1-g+\sum \{\mu_i / a\}} \prod_{i=1}^n \frac{\mu_i^{\lfloor \mu_i / a \rfloor}}{\lfloor \mu_i / a \rfloor !} \int_{\overline{\mathcal M}_{g, [-\mu]}({\mathcal B}\mathbb{Z}_a)} \frac{\sum_{i=0}^\infty (-a)^i \lambda_i^U}{\prod_{i=1}^n (1-\mu_i \overline{\psi}_i)}
\]
\end{theorem}
where $\overline{\mathcal{M}}_{g,\gamma}(\mathcal{B}\mathbb{Z}_a)$ is the moduli space of stable maps to $\mathcal{B}\mathbb{Z}_a$, the classifying stack of $\mathbb{Z}_a$ given by a point with trivial $\mathbb{Z}_a$ action, and $\lambda_i^U$ are generalisations of the Hodge class.

\begin{theorem} [\cite{BSLMMir,DLNOrb}]
For a fixed positive integer $a$, consider the rational spectral curve $C$ given by
\[
x(z) = z \exp(-z^a) \qquad \text{and} \qquad y(z) = z^a.
\]
The analytic expansion of the Eynard--Orantin invariant $\omega_{g,n}$ of $C$ around $x_1 = x_2 = \cdots = x_n = 0$ is given by
\[
\omega_{g,n} = \sum_{\mu_1, \ldots, \mu_n = 1}^\infty \frac{H_{g,n}^{[a]}(\mu_1, \ldots, \mu_n)}{(2g-2+n+\frac{|\mu|}{a})!} \prod_{i=1}^n x_i^{\mu_i-1} \dd x_i.
\]
\end{theorem}

\begin{definition}
For a fixed positive integer $a$, we define an $a$-fold {\em branching graph of type $(g; \mu)$} to be an edge-labeled fatgraph of type $(g, \ell(\mu))$ such that
\begin{itemize}
\item there are $\frac{|\mu|}{a}$ vertices and at each of them there are $am$ adjacent half-edges that are cyclically labeled
\[
1, 2, 3, \ldots, m, 1, 2, 3, \ldots, m, \ldots, 1, 2, 3, \ldots, m;
\]
\item there are exactly $m$ (full) edges that are labeled $1, 2, 3, \ldots, m$; and
\item  the $n$ faces are labeled and have perimeters given by $(\mu_1m, \mu_2m, \ldots, \mu_nm)$;
\item each face has a marked $m$-label (of the possible $\mu_k$ appearances of $m$.)
\end{itemize}
\end{definition}
Here, we take $m = 2g - 2 + \ell(\mu) + \frac{|\mu|}{a}$ due to the Riemann--Hurwitz formula.

\begin{proposition} \cite{DLNOrb}
The orbifold Hurwitz number $H_{g,n}^{[a]}(\mu_1, \mu_2, \ldots, \mu_n)$ is equal to the number of $a$-fold branching graphs of type $(g; \mu)$. 
\end{proposition}

\subsection{Pruned orbifold Hurwitz numbers}

One obtains pruned orbifold Hurwitz numbers by restricting the enumeration to the set of pruned orbifold branching graphs, which are obtained by introducing the same simple condition on vertex degrees.

\begin{definition}
We call an orbifold branching graph {\em pruned} if each vertex has essential degree at least two. Let $K_{g,n}^{[a]}(\mu_1, \mu_2, \ldots, \mu_n)$ be the number of pruned $a$-fold branching grahps of type $(g; \mu)$, where $\mu = (\mu_1, \mu_2, \ldots, \mu_n)$. Furthermore, let $m = m(g, \mu) = 2g-2+n+\frac{|\mu|}{a}$ and define the normalisation
\[
\widehat{K}_{g,n}^{[a]}(\mu_1, \mu_2, \ldots, \mu_n) = \frac{K_{g,n}(\mu_1, \mu_2, \ldots, \mu_n)}{m!}.
\]
\end{definition}


\begin{proposition}[Cut-and-join recursion for pruned orbifold Hurwitz numbers]
For $2g-2+n > 0$,
\begin{align*}
m(g,\mu) \, \widehat{K}_{g,n}^{[a]}(\mu_S) &= \sum_{i < j} \mu_i\mu_j \sum_{\alpha + a\beta = \mu_i + \mu_j + a} \beta \widehat{K}_{g,n-1}^{[a]}(\bm{\mu}_{S \setminus \{i,j\}}, \alpha) \\
&+ \frac{1}{2} \sum_{i=1}^n \mu_i\sum_{\alpha + \beta + \gamma = \mu_i + 1}  \gamma \left[ \widehat{K}_{g-1,n+1}^{[a]}(\bm{\mu}_{S \setminus\{i\}}, \alpha, \beta) + \mathop{\sum_{g_1+g_2=g}}_{I \sqcup J = S \setminus\{i\}}^{\mathrm{stable}} \widehat{K}_{g_1, |I|+1}^{[a]}(\mu_I, \alpha) ~ \widehat{K}_{g_2, |J|+1}^{[a]}(\mu_J, \beta) \right]
\end{align*}
\end{proposition}

\begin{proof}
The proof follows from removal of edges from branching graphs and is essentially the same as the proof of Proposition~\ref{th:cutandjoin}.
\end{proof}

\subsection{The pruning correspondence}

\begin{proposition} \label{th:pruningorb}
For $(g,n) \neq (0,1)$,
\[
\widehat{H}_{g,n}(\mu_1, \ldots, \mu_n) = \sum_{\nu_1, \ldots, \nu_n = 1}^{\mu_1, \ldots, \mu_n} \widehat{K}_{g,n}(\nu_1, \ldots, \nu_n) \prod_{i=1}^n \frac{\mu_i^{\frac{\mu_i-\nu_i}{a}}}{(\frac{\mu_i-\nu_i}{a})!}
\]
\end{proposition}
\begin{proof}
As in the proof of Proposition~\ref{th:pruning} the factor
\[
\frac{(2g-2+n+\frac{|\bm{\mu}|}{a})!}{(2g-2+n+\frac{|\bm{\nu}|}{a})! \, (\mu_1-\nu_1)! \, \cdots \, (\mu_n-\nu_n)!}
\]
accounts for the number of ways to choose the set of edge labels for the underlying pruned branching graph as well as the set of $\frac{\mu_i - \nu_i}{a}$ edge labels to be added to face $i$ for $i = 1, 2, \ldots, n$.  The factor $\mu^{\frac{\mu-\nu}{a}}$ on each face generalises the $a=1$ case where now we let $T_{k,e}^{[a]}$ be the number of rooted forests with $k$ labeled components and $e$ labeled edges, counted with weight $a^{\# \text{internal edges}}$. Then
\[
T_{k,e}^{[a]} = k(ae+k)^{e-1}
\]
and $\mu^{\frac{\mu-\nu}{a}}=\frac{\mu}{\nu}T_{\nu,\frac{\mu-\nu}{a}}^{[a]}.$
\end{proof}

\begin{proposition} \label{th:expansionorb}
The expansions of the $a$-fold Hurwitz differentials at $z_1 = z_2 = \cdots = z_n = 0$ satisfy
\[
\omega_{g,n} = \sum_{\mu_1, \ldots, \mu_n = 1}^\infty \widehat{K}_{g,n}^{[a]}(\mu_1, \mu_2, \ldots, \mu_n) \prod_{i=1}^n z_i^{\mu_i-1} \dd z_i, \qquad \text{for } 2g-2+n>0.
\]
\end{proposition}

\begin{proof}
Recall that
\[
\omega_{g,n} = \sum_{\mu_1, \ldots, \mu_n = 1}^\infty \widehat{H}_{g,n}^{[a]}(\mu_1, \mu_2, \ldots, \mu_n) \prod_{i=1}^n  x_i^{\mu_i-1} \dd x_i,
\]
and define
\[
\overline{\omega}_{g,n} = \sum_{\mu_1, \ldots, \mu_n = 1}^\infty \widehat{K}_{g,n}^{[a]}(\mu_1, \mu_2, \ldots, \mu_n) \prod_{i=1}^n z_i^{\mu_i-1} \dd z_i.
\]

We will show that $\overline{\omega}_{g,n} = \omega_{g,n}$ for $2g-2+n > 0$ by calculating the following residue.
\begin{align*}
\mathop{\text{Res}}_{x_1=0} \cdots \mathop{\text{Res}}_{x_n=0} \overline{\omega}_{g,n} \prod_{i=1}^n x_i^{-\mu_i} &= \mathop{\text{Res}}_{z_1=0} \cdots \mathop{\text{Res}}_{z_n=0} \sum_{\nu_1, \ldots, \nu_n = 1}^\infty \widehat{K}_{g,n}^{[a]}(\nu_1, \ldots, \nu_n) \prod_{i=1}^n  z_i^{\nu_i-1} \dd z_i \left[ z_i \exp(-z_i^a) \right]^{-\mu_i} \\
&= \mathop{\text{Res}}_{z_1=0} \cdots \mathop{\text{Res}}_{z_n=0} \sum_{\nu_1, \ldots, \nu_n = 1}^\infty \widehat{K}_{g,n}^{[a]}(\nu_1, \ldots, \nu_n) \prod_{i=1}^n z_i^{\nu_i-1} \dd z_i \, z_i^{-\mu_i} \sum_{m_i=0}^\infty \frac{\mu_i^{am_i}}{m_i!} z_i^{am_i} \\
&= \sum_{\nu_1, \ldots, \nu_n = 1}^{\mu_1, \ldots, \mu_n} \widehat{K}_{g,n}^{[a]}(\nu_1, \ldots, \nu_n) \prod_{i=1}^n \mathop{\text{Res}}_{z_i=0} z_i^{\nu_i-1} \dd z_i \, z_i^{-\mu_i} \sum_{m_i=0}^\infty \frac{\mu_i^{am_i}}{m_i!} z_i^{am_i} \\
&= \sum_{\nu_1, \ldots, \nu_n = 1}^{\mu_1, \ldots, \mu_n} \widehat{K}_{g,n}^{[a]}(\nu_1, \ldots, \nu_n) \prod_{i=1}^n\frac{\mu_i^{\frac{\mu_i-\nu_i}{a}}}{(\frac{\mu_i-\nu_i}{a})!} \\
&= \widehat{H}_{g,n}^{[a]}(\mu_1, \ldots, \mu_n)
\end{align*}
It follows that
\[
\overline{\omega}_{g,n} = \sum_{\mu_1, \ldots, \mu_n = 1}^\infty \widehat{H}_{g,n}(\mu_1, \mu_2, \ldots, \mu_n) \prod_{i=1}^n \mu_i x_i^{\mu_i-1} \dd x_i = \omega_{g,n}. \qedhere
\]
\end{proof}

Recall that
\[
H_{g,n}^{[a]}(\mu_1, \ldots, \mu_n) = a^{1-g+d/a} \prod_{i=1}^n \frac{(\mu_i/a)^{\lfloor \mu_i/a \rfloor}}{\lfloor \mu_i/a \rfloor!} \times Q_{g,n}^{[a]}(\mu_1, \ldots, \mu_n) = a^{1-g+\sum \{\mu_i/a\}} \prod_{i=1}^n \frac{\mu_i^{\lfloor \mu_i/a \rfloor}}{\lfloor \mu_i/a \rfloor !} \times Q_{g,n}^{[a]}(\mu_1, \ldots, \mu_n)
\]

\begin{proposition} 
For a fixed positive integer $a$ and $2g-2+n > 0$, the normalised pruned orbifold Hurwitz number $\widehat{K}_{g,n}^{[a]}(\mu_1, \mu_2,  \ldots, \mu_n)$ is a quasi-polynomial modulo $a$ in $\mu_1, \mu_2, \ldots, \mu_n$ of degree $6g-6+3n$.
\end{proposition}

\begin{proof}
One can prove this in an analogous way to Proposition~\ref{th:polynomial} but instead we will use the spectral curve.  By Proposition~\ref{th:expansionorb},
$$\omega_{g,n} = \sum_{\mu_1, \ldots, \mu_n = 1}^\infty \widehat{K}_{g,n}^{[a]}(\mu_1, \mu_2, \ldots, \mu_n) \prod_{i=1}^n z_i^{\mu_i-1} \dd z_i$$
is a meromorphic multidifferential on the curve $(x(z),y(z)) = (z \exp(-z^a), z^a)$ and hence it is rational in $z$.  Furthermore, by the general theory of Eynard-Orantin invariants it has poles only at the zeros of $dx_i$, hence when $0=dx_i=(1-z_i^a)\exp(-z_i^a)dz_i$, i.e.\ at the $a$th roots of unity $z_i^a=1$. A rational function in $z$ with poles only at $\{z\mid z^a=1\}$ has an expansion $\sum p(n)z^n$ around $z=0$ for $p(n)$ a quasi-polynomial mod $a$ meaning it is polynomial on each coset of the finite index sublattice $a\bz^n$ of  $\bz^n$.  Its degree follows from the order of the poles of $\omega_{g,n}$ which is $6g-4+2n$ again by a general property of Eynard-Orantin invariants.

\end{proof}

\section{Belyi Hurwitz numbers}

For $n>0$ and $g\geq 0$ define the set of Belyi Hurwitz covers:
\begin{align*}
Z_{g,n}(\mu)=\{f:\Sigma\to S^2\mid& \Sigma \text{ connected genus } g \text{ unramified over } S^2-\{0,1,\infty\};\\
&f^{-1}(\infty)=(p_1,...,p_n) \text{ with respective ramification }\mu=(\mu_1,...,\mu_n);\\
& \quad\text{ ramification }(2,2,...,2) \text{ over }1; \quad\text{ arbitrary ramification over }0 \} /\sim
\end{align*}
where $\{f_1: \Sigma_1 \to \mathbb{CP}^1\}\sim\{f_2: \Sigma_2 \to \mathbb{CP}^1\}$ if there exists $h: \Sigma_1 \to \Sigma_2$ that satisfies $f_1 = f_2 \circ h$ and preserves the labels over $\infty$.

Define the Belyi Hurwitz numbers:
$$M_{g,n}(\mu_1,...,\mu_n)=\sum_{f\in Z_{g,n}(\mu)}\frac{1}{|\text{Aut\ }f|}.$$
 
Now define the set of pruned Belyi Hurwitz covers:
\begin{align*}
Z^0_{g,n}(\mu)=\{f\in Z_{g,n}(\mu)\mid& \text{ all points in }f^{-1}(0) \text{ have nontrivial ramification}\} 
\end{align*}
and the corresponding pruned Belyi Hurwitz numbers:
$$N_{g,n}(\mu_1,...,\mu_n)=\sum_{f\in Z^0_{g,n}(\mu)}\frac{1}{|\text{Aut\ }f|}.$$
 A recursion expressing $M_{g,n}$ in terms of $M_{g',n'}$ uses a cut and join argument known as Tutte's recursion in the planar case \cite{TutCen} and more generally arises out of matrix integral expansions \cite{BIZQua,EOrTop}.  See also \cite{DMSSSpe} where $M_{g,n}(\mu)$ is treated as a generalised Catalan number.  Recursions expressing $N_{g,n}$ in terms of $N_{g',n'}$ were given in \cite{NorCou}.
 
To any $f\in Z_{g,n}(\mu)$ one can associate a fatgraph $\Gamma_f=f^{-1}[0,1]\subset\Sigma$, meaning that $\Sigma-f^{-1}[0,1]$ is a union of disks, or equivalently a discrete surface of genus $g$ obtained by gluing together $n$ polygonal faces of perimeters $\mu_1,...,\mu_n$.    Equivalently, a fatgraph is described by its set of oriented edges $X$ equipped with automorphisms $\tau_i:X\to X$, $i=0,1$.  Then $\Gamma_f=(X_f,\tau_0,\tau_1)$ where $X_f=f^{-1}(0,1)$, $\tau_0:X_f\to X_f$ is the monodromy map around 0; $\tau_1:X_f\to X_f$ is the monodromy map around 1.  The vertices of the fatgraph or polygonal faces correspond to $V_f=f^{-1}(0)\cong X_f/\tau_0$ and the edges correspond to $E_f= X_f/\tau_1\cong f^{-1}(1)$.  The boundary components correspond to $X_f/\tau_2\cong f^{-1}(\infty)$ for $\tau_2=\tau_0\tau_1$ and its length is the size of the orbit of $\tau_2$.  An automorphism of a fatgraph $\Gamma=(X,\tau_0,\tau_1)$ is a map $g:X\to X$ that commutes with $\tau_0$ and $\tau_1$.  From the fatgraph one can reconstruct the map $f$.  Hence the Belyi Hurwitz numbers can be equivalently defined as follows:
$$M_{g,n}(\mu_1,...,\mu_n)=\sum_{\Gamma\in \fat(\mu)}\frac{1}{|\text{Aut\ }\Gamma|}$$
where $\fat(\mu)$ is the set of all genus $g$ fatgraphs with $n$ labeled boundary components of respective lengths $(\mu_1,...,\mu_n)$.  Similarly
$$N_{g,n}(\mu_1,...,\mu_n)=\sum_{\Gamma\in \fat^0(\mu)}\frac{1}{|\text{Aut\ }\Gamma|}$$
where $\fat^0(\mu)\subset\fat(\mu)$ consists of those fatgraphs with no valence 1 vertices---pruned fatgraphs.  It is this graph representation that justifies the term "pruned" Belyi Hurwitz number. 

Let $\modm_{g,n}$ be the moduli space of genus $g$ curves with $n$ labeled points.  For each $\mu=(\mu_1,..,\mu_n)$ there is the Penner-Harer-Mumford-Thurston cell decomposition  
\begin{equation}  \label{cell}
\modm_{g,n}\cong\bigcup_{\Gamma\in \fat}P_{\Gamma}(\mu_1,..,\mu_n)
\end{equation}
where the indexing set $\fat$ is the space of fatgraphs with all vertices of valence $\geq 3$, of genus $g$ and $n$ labeled boundary components.  The cell decomposition \eqref{cell} arises by the existence and uniqueness of Strebel differentials on a compact Riemann surface $\Sigma$ with $n$ labeled points $(p_1,...,p_n)$ and $n$ positive real values $(\mu_1,..,\mu_n)$.  A Strebel differential is a meromorphic quadratic differential $\omega$, holomorphic on $\Sigma-(\mu_1,..,\mu_n)$.  Any quadratic differential gives rise to vertical and horizontal foliations along which $\omega$ is real.  Along the horizontal and vertical foliations $\omega$ is real and positive, respectively negative.   In terms of a local coordinate $z$ away from zeros and poles one can write $\omega=\dd z^2=\dd x^2-\dd y^2+2i\dd x\dd y$ which is real and positive along $y=$ constant and negative along $x=$ constant.  A Strebel differential is distinguished by the fact that its horizontal foliation has compact leaves and its poles occur at the $p_k$ with principal part $\mu_k\dd z/z^2$.   It has one unique singular compact leaf which is a labeled fatgraph with lengths on edges.  The important point is that this singular compact leaf has no valence 1 vertices.  A valence 1 vertex corresponds to a singularity of the form $\dd z^2/z$ which is prohibited on Strebel differentials.  In summary, the Strebel differentials give rise to pruned fatgraphs with lengths on edges and no valence 2 vertices.  

The natural map $Z_{g,n}(\mu)\to\modm_{g,n}$ that sends $f:\Sigma\to S^2$ to its domain curve $(\Sigma,p_1,...,p_n)$, where $f(p_k)=(\infty)$, can be combined with the cell decomposition \eqref{cell} using the same $\mu$ to assign to $f$ a fatgraph $\Gamma^f$ with no valence 2 vertices.  In general, $\Gamma^f\neq\Gamma_f$.

Underlying $\Gamma_f$ is a fatgraph $\tilde{\Gamma}_f$ with no valence 2 vertices, essentially obtained by ignoring valence 2 vertices of $\Gamma_f$.  On the level of oriented edges $X_f$ and $\tilde{X}_f$, there are maps $\pi:X_f\to\tilde{X}_f$ and $\iota:\tilde{X}_f\to X_f$ satisfying $\pi\circ\iota=id$, $\pi\circ\tau_1=\tau_1\circ\pi$ and $\iota\circ\tau_0=\tau_0\circ\iota$.  The induced map $\pi_*:E_f\to\tilde{E}_f$ is surjective and one-to-one except on edges adjacent to valence 2 vertices, and $\iota_*:\tilde{V}_f\to V_f$ is injective with image all of $V_f$ except for valence 2 vertices.  For general $f\in Z_{g,n}(\mu)$, $\tilde{\Gamma}_f\neq\Gamma^f$ since $\Gamma_f$ usually has valence 1 vertices.  However 
$$f\in Z^0_{g,n}(\mu)\quad\Rightarrow\quad\tilde{\Gamma}_f=\Gamma^f.$$  
In other words $Z^0_{g,n}(\mu)$ sits naturally inside $\modm_{g,n}$ and this gives rise to a third description of $N_{g,n}(\mu_1,...,\mu_n)$ as the number of integral points inside rational polytopes making up the cells of $\modm_{g,n}$.  The cells of \eqref{cell} are compact convex polytopes 
$$P_{\Gamma}(\mu_1,..,\mu_n)=\{{\bf x}\in\br_+^{E(\Gamma)}|A_{\Gamma}{\bf x}=\mu\}$$
where $\mu=(\mu_1,..,\mu_n)\in\br^n$ and $A_{\Gamma}:\br^{E(\Gamma)}\to\br^n$ is the incidence matrix that maps an edge to the sum of its two incident boundary components.  Define $N_{\Gamma}(\mu_1,..,\mu_n)=\#\{\bz_+^{E(\Gamma)}\cap P_{\Gamma}(\mu_1,..,\mu_n)\}$.
Then
$$N_{g,n}(\mu_1,...,\mu_n)=\sum_{\Gamma\in \fat}\frac{1}{|Aut \Gamma|}N_{\Gamma}(\mu_1,...,\mu_n).$$
An important consequence of this realisation of $N_{g,n}(\mu)$ as counting integral points is the identity
$$N_{g,n}(0,...,0)=\chi(\modm_{g,n}).$$  One makes sense of evaluation of $N_{g,n}$ at $(0,...,0)$ by using the fact that $N_{g,n}(\mu_1,...,\mu_n)$ is quasi-polynomial in the $\mu_i$.

\subsection{The pruning correspondence}

\begin{proposition}[\cite{NScPol}]
\begin{equation}  \label{prunebelyi}
M_{g,n}(\mu_1, \ldots, \mu_n) \prod_{i=1}^n \mu_i = \sum_{\nu_1, \ldots, \nu_n = 1}^{\mu_1, \ldots, \mu_n} N_{g,n}(\nu_1, \ldots, \nu_n) \prod_{i=1}^n \nu_i \binom{\mu_i}{\frac{\mu_i-\nu_i}{2}}
\end{equation}
\end{proposition}

\begin{proof}
\begin{align*}
M_{g,n}(\mu_1, \ldots, \mu_n) \prod_{i=1}^n \mu_i &= \mathop{\text{Res}}_{x_1=\infty} \cdots \mathop{\text{Res}}_{x_n=\infty} \omega_{g,n} \prod_{i=1}^n x_i^{\mu_i} \\
&= \mathop{\text{Res}}_{z_1=0} \cdots \mathop{\text{Res}}_{z_n=0} \sum_{\nu_1, \ldots, \nu_n = 1}^\infty N_{g,n}(\nu_1, \ldots, \nu_n) \prod_{i=1}^n \nu_i z_i^{\nu_i-1} \dd z_i x_i^{\mu_i} \\
&= \mathop{\text{Res}}_{z_1=0} \cdots \mathop{\text{Res}}_{z_n=0} \sum_{\nu_1, \ldots, \nu_n = 1}^\infty N_{g,n}(\nu_1, \ldots, \nu_n) \prod_{i=1}^n \nu_i z_i^{\nu_i-1} \dd z_i \left( z_i + \frac{1}{z_i} \right)^{\mu_i} \\
&= \mathop{\text{Res}}_{z_1=0} \cdots \mathop{\text{Res}}_{z_n=0} \sum_{\nu_1, \ldots, \nu_n = 1}^\infty N_{g,n}(\nu_1, \ldots, \nu_n) \prod_{i=1}^n \nu_i z_i^{\nu_i-1} \dd z_i \sum_{k_i=0}^{\mu_i} \binom{\mu_i}{k_i} z_i^{\mu_i-2k_i} \\
&= \sum_{\nu_1, \ldots, \nu_n = 1}^{\mu_1, \ldots, \mu_n} N_{g,n}(\nu_1, \ldots, \nu_n) \prod_{i=1}^n \nu_i \binom{\mu_i}{\frac{\mu_i+\nu_i}{2}} \\
&= \sum_{\nu_1, \ldots, \nu_n = 1}^{\mu_1, \ldots, \mu_n} N_{g,n}(\nu_1, \ldots, \nu_n) \prod_{i=1}^n \nu_i \binom{\mu_i}{\frac{\mu_i-\nu_i}{2}} 
\end{align*}
hence \eqref{prunebelyi} follows.

The significance of this result here is that one can also give a combinatorial proof which simply formalises the fact that, given a fatgraph, one can repeatedly remove vertices of degree 1 and their incident edges to obtain a pruned fatgraph in a unique way.  Hence this gives another example of  pruning.
\end{proof}
\begin{remark}
We could have naturally defined $M_{g,n}(\mu_1, \ldots, \mu_n)$ and $N_{g,n}(\mu_1, \ldots, \mu_n)$ to include a factor of  $\mu_1 \cdots \mu_n$ in which case \eqref{prunebelyi} would look much more like the pruning correspondence for simple Hurwitz numbers in Proposition~\ref{th:pruning} differing by a simple combinatorial factor.
\end{remark}

\section{Gromov-Witten invariants of $\bp^1$}  \label{sec:gw}

In this section we apply the idea of pruning to the Gromov-Witten invariants of $\bp^1$.  Unlike the previous sections, the aim here is to predict interesting structure and the problem is not yet resolved.

Assemble the Gromov-Witten invariants into the generating function
\[\Omega^g_n(x_1,...,x_n)=\sum_{\bf\mu}\left\langle \prod_{i=1}^n \tau_{\mu_i}(\omega) \right\rangle^g_d\cdot\prod_{i=1}^n(\mu_i+1)!x_i^{-\mu_i-2}dx_i.\]
\begin{equation}  \label{eq:lnz}
C=\begin{cases}x=z+1/z\\ 
y=\ln{z}\sim\sum\frac{(1-z^2)^k}{\hspace{-3mm}-2k}.
\end{cases}
\end{equation} 
$\omega^g_n$ of $(C,x,y_N)$ stabilises for $N\geq 6g-6+2n$, where $y_N$ are the partial sums for the expansion for $y$.
\begin{theorem}[\cite{DOSSIde,NScGro}]
For $2g-2+n>0$, the Eynard-Orantin invariants of the curve $C$ defined in (\ref{eq:lnz}) agree with the generating function for the Gromov-Witten invariants of $\bp^1$:
\[ \omega^g_n\sim\Omega^g_n(x_1,...,x_n).\]
More precisely, $\Omega^g_n(x_1,...,x_n)$ gives an analytic expansion of $\omega^g_n$ around a branch of $\{x_i=\infty\}$.
\end{theorem}
This was proven in \cite{NScGro} for the case $g=0,1$ and for all $g$ in \cite{DOSSIde}.  

The expansion of  $\omega^g_n$ around $z_i=0$
\[ \omega^g_n\sim\sum_{\mu}N_{g,n}(\mu)\prod \mu_i z_i^{\mu_i-1}dz_i\]
has coefficients $N_{g,n}(\mu)$ which are quasi-polynomial mod $2$.  The simplest quasi-polynomials are shown in the table.
\begin{table}[h!] 
\begin{spacing}{1.4}  
\begin{tabular}{||l|c|c|c||} 
\hline\hline

{\bf g} &{\bf n}&\# odd $\mu_i$&$N_{g,n}(\mu_1,...,\mu_n)$\\ \hline

0&3&0,2&$0$\\ \hline
0&3&1,3&$1$\\ \hline
1&1&0&$0$\\ \hline
1&1&1&$\frac{1}{48}(\mu_1^2-3)$\\ \hline
0&4&0,4&$\frac{1}{4}(\mu_1^2+\mu_2^2+\mu_3^2+\mu_4^2)$\\ \hline
0&4&1,3&$0$\\ \hline
0&4&2&$\frac{1}{4}(\mu_1^2+\mu_2^2+\mu_3^2+\mu_4^2-2)$\\ \hline
1&2&0&$\frac{1}{384}(\mu_1^2+\mu_2^2-8)(\mu_1^2+\mu_2^2)$\\ \hline
1&2&1&$0$\\ \hline
1&2&2&$\frac{1}{384}(\mu_1^2+\mu_2^2-6)(\mu_1^2+\mu_2^2-2)$\\ \hline
2&1&0&$0$\\ \hline
2&1&1&$\frac{1}{2^{16}3^35}(\mu_1^2-1)^2(5\mu_1^4-186\mu_1^2+1605)$\\
\hline\hline
\end{tabular} 
\end{spacing}
\end{table}

The quasi-polynomials satisfy the following relations.
\begin{align*}
N_{g,n+1}(0,\mu_1,\dots \mu_n) &=\sum_{j=1}^n\sum_{k=1}^{\mu_j}kN_{g,n}(\mu_1,\dots,\mu_n)|_{\mu_j=k}
\\
N_{g,n+1}(1,\mu_1,\dots,\mu_n)
&=\sum_{j=1}^n\sum_{k=1}^{\mu_j}kN_{g,n}(\mu_1,\dots,\mu_n)|_{\mu_j=k}+\frac{\chi-|\mu|}{2}N_{g,n}(\mu_1,\dots, \mu_n)\end{align*}
for $\chi=2-2g-n$ and $\displaystyle|\mu|=\sum_{j=1}^n\mu_j$.  \\

A natural question is whether the quasi-polynomials $N_{g,n}$ obtained from the expansion of $\omega^g_n$ around $z_i=0$ yield an interesting and useful enumerative problem.  In all calculated cases the genus 0 invariants $N_{0,n}$ take integral values which lends evidence that there may be an underlying enumerative problem.

\subsection{Cycle Hurwitz problem.}
The following Hurwitz problem was introduced and studied by Okounkov and Pandharipande in \cite{OPaGro}.  Given $\{x_1,...,x_n\}\subset S^2$, define
\begin{align*}
\cc_{g,n}(\mu)=\{f:\Sigma\to S^2\mid& \Sigma \text{ connected genus } g, \text{ unramified over } S^2-\{x_1,...,x_n\};\\
&\hspace{3cm}\text{ ramification }(\mu_k,1,1,...,1) \text{ over }x_k \} /\sim.
\end{align*}
Define the cycle Hurwitz numbers:
\begin{equation*} 
P_{g,n}(\mu_1,...,\mu_n)=\sum_{f\in\cc_{g,n}(\mu)}\frac{1}{|\text{Aut\ }f|}.
\end{equation*}

\begin{theorem}[Okounkov-Pandharipande \cite{OPaGro}] 
The cycle Hurwitz numbers $P_{g,n}(\mu_1,...,\mu_n)$ contribute all stable maps with smooth domain curves to $\displaystyle \prod_{i=1}^n(\mu_i-1)!\left\langle \prod_{i=1}^n \tau_{\mu_i-1}(\omega)\right\rangle^g_d$.\quad 
\end{theorem}
In other words, Gromov-Witten invariants compactify the Hurwitz count by allowing stable domains.

\begin{lemma}  \label{th:cycle03}
$P_{0,3}(\mu_1,\mu_2,\mu_3)=1$
\end{lemma}
\begin{proof}
Denote by $C_{\mu_1}\subset S_d$ the conjugacy class in the symmetric group consisting of all permutations with cycle structure $(\mu_1,1,1...,1)$.  
The lemma is equivalent to the statement 
\begin{equation}  \label{fact}
\#\{(\sigma_1,\sigma_2,\sigma_3)\mid \sigma_i\in C_{\mu_i},\ \sigma_1\cdot\sigma_2\cdot\sigma_3=1\text{ is a transitive factorisation}\}=d!.
\end{equation}
To get the Hurwitz number we divide \eqref{fact} by $d!$ corresponding to identifying equivalent products
\[(\sigma_1,\sigma_2,\sigma_3)\sim(g\sigma_1g^{-1},g\sigma_2g^{-1},g\sigma_3g^{-1}),\quad g\in S_d\]  
or equivalently isomorphic branched covers.  If a product is fixed by conjugation then this defines an automorphism of the branched cover.

It remains to prove \eqref{fact}.  By the Riemann-Hurwitz formula the degree $d$ of the cover satisfies $\mu_1+\mu_2+\mu_3=2d+1$.  

We begin with an example.  Suppose $(\mu_1,\mu_2,\mu_3)=(d,d,1)$.  Then there is a unique cover with two totally ramified points,  and it has automorphism group of size $d$, leading to $1/d$.  Equivalently the number of (transitive) factorisations $\sigma_1\sigma_2=(1)$ is $(d-1)!$.  An extra factor of $d$ comes from the choice of the third point corresponding to $\sigma_3=1$---there are $d$ such choices---or equivalently the third point  makes the automorphism group trivial.  

More generally, identify $\sigma_i$ with its cycle of length $\mu_i$ (and if $\mu_i=1$ choose a single number to represent $(1)$ or ignore this case.)   In order that the factorisation is transitive and has product $(1)$, there is exactly one number common to all three cycles which we suppose to be 1.  Also suppose $\sigma_1=(1...\mu_1)$.  The numbers $\{2,...,a\}$ appear in exactly one of $\sigma_2$ and $\sigma_3$ and their location is uniquely determined.  Also, $\sigma_2$ and $\sigma_3$ both contain the numbers $\{1,a+1,a+2,...,d\}$ and the order of these numbers in $\sigma_2$ determines their order in $\sigma_3$.  Hence the number of transitive factorisations is
$$ \binom{d}{a}\cdot(a-1)!\cdot a\cdot (d-a)!=d!$$
where the factor $\binom{d}{a}$ chooses the elements of $\sigma_1$, the factor $(a-1)!$ chooses the cycle $\sigma_1$, the factor $a$ chooses the number common to all three factors and the factor $(d-a)!$ chooses the order of $\{a+1,a+2,...,d\}$ in $\sigma_2$.
\end{proof}

\begin{corollary}  \label{th:N=P}
If $\mu_1$, $\mu_2$ and $\mu_3$ satisfy the triangle inequalities then
$N_{0,3}(\mu_1,\mu_2,\mu_3)=P_{0,3}(\mu_1,\mu_2,\mu_3).$
\end{corollary}
\begin{proof}
We know that $N_{0,3}(\mu_1,\mu_2,\mu_3)=1$ iff $\sum\mu_i$ is odd, so the point of this corollary is simply to identify this appearance of 1 with the appearance of 1 in Lemma~\ref{th:cycle03}.  The triple $(\mu_1,\mu_2,\mu_3)$ appears in a Hurwitz problem if their sum is odd, since $\mu_1+\mu_2+\mu_3=2d+1$, and if $\mu_i\leq d$, $i=1,2,3$.  But 
$$\mu_1\leq d\Leftrightarrow -\mu_1+\mu_2+\mu_3>0$$
which is one of the three triangle inequalities.  The other two triangle inequalities hold by the same argument.
\end{proof}
In general $N_{g,n}(\mu_1,...,\mu_n)\neq P_{g,n}(\mu_1,...,\mu_n)$.  For example, $N_{0,3}(2d-1,1,1)=1$ whereas $P_{0,3}(2d-1,1,1)=0$ for $d>1$.  Nevertheless, Corollary~\ref{th:N=P} suggests that $P_{g,n}(\mu_1,...,\mu_n)$ may equal $N_{g,n}(\mu_1,...,\mu_n)$ under conditions on $(\mu_1,...,\mu_n)$ and more generally $N_{g,n}(\mu_1,...,\mu_n)$ may be realised as the solution to a generalisation of the cycle Hurwitz problem involving stable curves for domains.

\end{document}